\documentclass{article}
\usepackage[utf8]{inputenc}

\newcounter{rmk}

\usepackage{amsmath}
\usepackage{float}

\usepackage{comment}
\usepackage[linktocpage=true]{hyperref}
\usepackage[margin=1.2 in, top=1 in, bottom= 1.2 in]{geometry}

\usepackage{indentfirst} 
\usepackage[english]{babel}
\usepackage{amsfonts}
\usepackage{mathrsfs}
\usepackage[mathscr]{euscript}
\usepackage{bbm}
\usepackage{latexsym}
\usepackage{mathdots}
\usepackage{amssymb}
\usepackage{mathtools}
\usepackage{graphicx}
\usepackage{tikz}
\usepackage{placeins}
\usepackage{esint}

\usepackage{enumitem}
\setlist{nolistsep}
\usepackage{amsthm}
\usepackage{relsize}
\usepackage{nicefrac}
\usepackage[capitalize]{cleveref}

\newtheoremstyle{plain}{3mm}{3mm}{\slshape}{}{\bfseries}{.}{.5em}{}
\newtheoremstyle{definition}{2mm}{2mm}{}{}{\bfseries}{.}{.5em}{}
\theoremstyle{plain}
	
\newtheorem{theorem}{Theorem}

\newtheorem{lemma}[theorem]{Lemma}

\newtheorem{corollary}[theorem]{Corollary}

\theoremstyle{definition}

\newtheorem{remark}[rmk]{Remark}

\theoremstyle{plain}
\newtheorem*{namedthm}{\namedthmname}
\newcounter{namedthm}
	

\newcommand{\R}{\mathbb{R}}

\newcommand{\D}{\mathcal{D}}
\newcommand{\Gee}{\mathcal{G}}

\newcommand{\eps}{\epsilon}

\newcommand{\A}{\mathcal{A}}
\newcommand{\B}{\mathcal{B}}
\newcommand{\G}{\mathcal{G}}

\title{Pointwise bounds on Dirichlet Green's functions for a singular drift term.}
\author{Aritro Pathak}

\date{}

\begin{document}
\maketitle

\begin{abstract}
We introduce a technique, to obtain pointwise upper and lower bounds for the Green’s function of elliptic operators whose principal part is the Laplacian and that include a drift term diverging near the boundary like a power of the inverse distance with exponent less than~$1$, in the unit ball~$B(0,1)\subset\mathbb{R}^n$, $n\ge3$. The constants in the upper estimates are uniform in~$B(0,r)$ for each~$r<1$, with explicit dependence on~$r$. Further, more generally, they may be majorized by a function radially integrable up to the boundary. These appear to be the first estimates for such non-coercive drifts and remain new even for smooth drifts, suggesting extensions to singular potentials as well.
\end{abstract}

\maketitle
\section{Introduction.}

Let $\Omega=B(0,1)$ be the unit ball in $\R^{n}$ with $n\geq 3$ . We denote, $\delta(X)=\text{dist}(X,\partial\Omega)$. The coordinates in $\R^{n}$ are written as $(x_1,\dots,x_n)$. In this paper, we consider in $\R^{n}$ the linear second order operator with a singular drift term, given for some $0<\beta<1$, 
\begin{equation}\label{Laplacian}
    L u=-\Delta u-\B\cdot \nabla u =0, \ \ \   (\delta(X))^{1-\beta} |\B(X)|\leq M.
\end{equation}
 Here $0< M<\infty$ .

We find pointwise upper and lower estimates for the Green's function for such an operator when the Green's function is locally in $C^{3,\alpha}_{loc}$ for some $\alpha>0$. 
   
 We require the pointwise condition that, 
     \begin{align}\label{nonpositivee}
         -\frac{M}{\delta(x)^{2-\beta}}\leq \nabla\cdot \B\leq 0, \ 
     \end{align}
In particular this also means that $\nabla\cdot\B$ is negative in the distributional sense, that is,
 \begin{align}\label{nonpositivee2}
       \int_{\Omega} (\B\cdot \nabla v)\geq 0, \ \forall v\in C^{\infty}_{0}(\Omega)\  \text{where}\  v\geq 0,
     \end{align}
     
     Note that when the drift satisfies the pointwise estimate of \cref{nonpositivee}, then according to Theorem 1.2 of \cite{Ha24}, solutions to the adjoint equation $L_T u=-\nabla\cdot(\nabla u+\B u)=0$ exists, and one considers the Green's functions $\G(x,y),\G_T (x,y)$  corresponding to $L$ and $L_T$ respectively. Further, we also have $\G(x,y)=\G_T(y,x)$.\footnote{We have stated the bounds as above for simplicity, but in general it is enough to consider $|\B|^2$ and $|\nabla\cdot \B|$ to belong respectively to the more general Morrey spaces $M^{\frac{n}{2-2\beta}}$ and $M^{\frac{n}{2-\beta}}$ as considered in \cite{Ha24}.}

In \cite{sssz12},  operators are considered that have the Laplacian term along with certain divergence free drift terms with similar scale-invariant properties, along with the parabolic counterparts:
\begin{align}\label{lap}
    \partial_{t}u -\B\cdot \nabla u -\Delta u=0
\end{align}

For the corresponding parabolic operator, with divergence free measurable drift terms ($\nabla\cdot \B=0$), an old result of Nash \cite{Na58} shows the pointwise upper bound:
\begin{align*}
    |\Gee(x,t;y,s)|\leq \frac{C}{(t-s)^{\frac{n}{2}}}.
\end{align*}
For the case where $$\nabla\cdot \B=0$$ in \cref{lap}, in \cite{sssz12} several regularity results are presented when writing $\B=\nabla \cdot\ d$ for an anti-symmetric matrix $d$ and considering the anti-symmetric part of the matrix $I+d$ to belong a BMO space with the symmetric part of $I+d$ being  comparable to the identity matrix $I$ up to a positive constant .In case the $\nabla\cdot \B=0$ condition is relaxed, one is also referred to the regularity results in \cite{Naz12} where the condition is relaxed to $\nabla\cdot \B=0\leq 0$. More recently, one has referred to the work in \cite{KimSa} for pointwise upper bounds on Green's function in bounded domains for the operator with a more general elliptic principal term and where the lower-order terms belong to certain $L^p$ spaces for $p\geq n$, and then further work in the setting of not necessarily bounded domains, of Sakellaris \cite{Sake} for lower-order terms in certain weak $L^{n,q}$ spaces, and of Mourgoglou \cite{Morg} for the lower order terms in Stummel-Kato spaces and variants of it. For Green's function corrsponding to drifts diverging in points within the domain, one is referred to \cite{MM22}.

It is seen that in general bounded domains, our singular drift term does not belong to either the weak $L^p$ space considered in \cite{Sake}, nor the Stummel-Kato spaces considered in Section 6 of \cite{Morg}. We note that the definition of Morrey spaces considered in \cite{Ha24} is more general than the definition considered in \cite{Morg}, where one requires the Morrey type estimate to hold in balls centered on the boundary as well, and our drift does satisfy the definition of \cite{Ha24}.


In this paper, the drift term is singular and behaves like an exponent of the inverse of the distance to the boundary, the exponent being strictly less than $1$. While we work specifically for the case of a drift that is pointwise bounded by $(1/\delta(x)^{1-\beta})$, the method works as long as we have a drift that is majorized by a radial function that is radially integrable up to the boundary, for the geometry of the unit ball considered here. Thus, locally the drift is merely radially integrable up to the boundary, and not required to be in $L^n$ or $L^{n,p}$ spaces.

The results here correspond to a situation where the corresponding bilinear form is not coercive, and thereby expected energy estimates fail. This appears to be the first instance of such pointwise estimates on the Green's function in such a setting.

More generally, we can define the operator:
\begin{equation}\label{imp22}
    L_1 u=-\nabla\cdot(\A\nabla u) +\B\cdot \nabla u =0, \ \ \  
\end{equation}

where have, some constants $\infty> \Lambda\geq \lambda>0$ , and $M$ as before, so that, 
\begin{align}\label{imp23}
    \A_{ij}(X) \eta_i \eta_j \geq \lambda |\eta|^2, |\A_{ij}(X)\eta_i \psi_j|\leq \Lambda |\eta| |\psi|, \ \ \text{for all} \ \eta,\psi\in \R^{n}\setminus \{0\}, \ \ \text{and} \ \delta(X)^{\beta}|\B(X)|\leq M.
\end{align}




We require the Green's function $G(\cdot,Y)$ for the operator \cref{Laplacian} to be in $C^{3,\alpha}_{loc}$ for some $\alpha>0$, in order to derive the pointwise upper bounds for the operator $L$. By the Schauder estimates, it is enough to consider the drift term $|\B|$ to be locally in $C^{0,\alpha}$ for an arbitrary $\alpha>0$. See for example, Theorem 6.10 of \cite{Gt}.




Here we state the three main results of this paper;

\begin{theorem}\label{thm1}
    For the elliptic operator in \cref{imp22} with the coefficients satisfying \cref{imp23}, we have the bound for the Green's function for the operator in \cref{imp22}: for any $z,y\in \Omega$ with $|z-y|\leq \frac{1}{2}\delta(y):=\frac{1}{2}\text{dist}(y,\partial\Omega)$ we have 
    \begin{equation}\label{lower}
        \G(y,z)\geq K(M,\lambda)\frac{1}{|z-y|^{n-2}}.
    \end{equation}
\end{theorem} 
\begin{theorem}\label{thm2}
    Consider the operator of \cref{Laplacian} and the Dirichlet Green's function corresponding to this operator. Further suppose that $\B\in C^{1,\alpha}(\Omega')$ for some $\alpha>0$, for any compactly supported subdomain in $\Omega$, and that $\B$ satisfies  \cref{nonpositivee2}.  For any $x\in \Omega$ with $|x|\leq \frac{1}{2}$, we have for some constant $K'$ dependent on $M, \lambda$ only,
    \begin{equation}
        \G(x,0)\leq K'(M,\lambda)\frac{1}{|x|^{n-2}}.
    \end{equation}
\end{theorem}
Here, the pole of Green's function has been considered to be only at the origin.

In case the solution to the adjoint equation also exists, we can use \cref{thm2} and some standard arguments to show that,
\begin{theorem}\label{thm3}
    Consider the operator of \cref{Laplacian} and the Dirichlet Green's function corresponding to this operator. Further suppose that $\B\in C^{1,\alpha}(\Omega')$ for some $\alpha>0$, for any compactly supported subdomain in $\Omega$, and that $\B$ satisfies  \cref{nonpositivee2}.  Then for any $y\in B(0,r)$ with $r<1$, and for any $x\in \Omega$ with $|x-y|\leq \frac{1}{2}\delta(y)$, we have for some constant $K'$ dependent on $M, \lambda$ only,
    \begin{equation}
        \G(x,y)\leq K'(M,\lambda,r)\frac{1}{|x-y|^{n-2}}.
    \end{equation}
    Here the bound $K'(M,\lambda,r)$ depends on $r$ as well.
\end{theorem}

We prove \cref{thm1} in Section 3, and \cref{thm2,thm3} in Section 4.

\begin{remark}
    We can consider potential terms with coefficients $\D$ that are bounded in magnitude by $M/(\delta(X))^{2}$, and in that case, we would need the kind of negativity condition:
\begin{align*}\label{nonpositive}
       \int_{\Omega} (\mathcal{D} v -\B\cdot \nabla v)\leq 0, \ \forall v\in C^{\infty}_{0}(\Omega),
\end{align*}
    (such as equation 8.8 of \cite{Gt} or equation 1.5 of \cite{Morg}) as in Section 4 of this paper, in particular for the additional term that would appear in \cref{eq12}. In presence of the $\D$ term, the key modification will be in \cref{diffineq}.  We leave the question of such potential terms, as well as applications of our technique for the landscape function, for future work. One is referred to \cite{An86,She94,She95,She99, Pog24} for related questions.

\begin{remark}
    As noted earlier, while for simplicity we have written the arguments for the specific drift considered in \cref{Laplacian}, the arguments of Section 3 and Section 4, can be naturally adopted for drifts that are majorized by a radial function integrable up to the boundary. In that case, as mentioned earlier, it is enough to consider that $|\B|^{2}$, $|\nabla\cdot\B|$  respectively belong to the general Morrey spaces $M^{\frac{n}{2-2\beta}}$ and $M^{\frac{n}{2-\beta}}$ considered in \cite{Ha24}.
\end{remark}    

\end{remark}

Sections 3 and 4, contain the main results of this paper. The result of Section 3 is short and similar to the proof in \cite{GrWi} for the case of $|\B|=0$. We also note that for the case of $\beta=0$, we have a counterexample to the existence of solutions, and pointwise upper bounds on the Green's function fail to exist \cite{Pat24}. There also, the domain is the unit ball, and the pole of the Green's function is at the origin.

We will also note that in \cref{lemma8,lemma9}, we see the distinction in the behavior of the Green's function when the gradient is small, in comparison to the case where the gradient is large. In this instance, it is actually easier to estimate the decay of the Green's function in case the gradient is small, as is done in the argument after the completion of \cref{lemma9}. In other contexts, this dichotomy of the behavior of solutions to elliptic equations where the gradient is large, versus where it is small, is well studied; see for example, \cite{Moo15}, and the references therein.

The basic novelties of the proof for the upper bound in Section 4 are in introducing a modified Lorentz norm when restricting to values of the Green's function within a ball containing the pole and separated from the boundary, and then to finally reduce the problem to looking at the level sets of the Green's function, using the spherical form of the Laplacian and the points at minimal and maximal distance of these level sets from the pole.

While we have worked specifically for the drift of \cref{Laplacian}, the method introduced here will be used in several directions in the future, as outlined in Section 5.

\section{Properties of the Green's function.}

We adopt the methods of Chapter 8 of \cite{Gt} .  The arguments of this section are standard.

Following the notation of \cite{Gt}, when $u$ is only weakly differentiable, then in a weak or generalized sense, $u$ is said to satisfy $L_1 u=0 (\geq 0, \leq 0 )$ in $\Omega$ if 
\begin{align}
    \mathcal{L}_1(u,v)=\int_{\Omega}\Big((\A_{ij} \partial_j u )(\partial_i v) -\B_i(\partial_i u) v \Big)dx=0 (\leq 0, \geq 0 ).
\end{align}


for all non-negative functions $v\in C^{1}_{0}(\Omega)$ which is the space of compactly supported functions whose first derivatives are continuous in $\Omega$.

Consider the $f,g^{i},\ i=1,2,\dots,n$ , locally integrable functions in $\Omega$. For our purposes of constructing the Green's function, it is enough to consider compactly supported $f,g^{i}$, $i=1,2,\dots,n$. Then a weakly differentiable function $u$ is called a weak or generalized solution of the inhomogeneous equation 
\begin{align}\label{eq122}
    L_1 u =f -\partial_i g^{i},
\end{align}
in $\Omega$ if 
\begin{align}\label{inhomo}
    \mathcal{L}_1 (u,v)=\int_{\Omega} (g^{i}\partial_i v +fv) dx=F(v),\ \ \forall v\in C^{1}_{0}(\Omega).
\end{align}

We say that a function $u$ is a solution of the generalized Dirichlet problem: $L_1 u =f +\partial_i g^{i}, u=\phi$ on $\partial\Omega$ if $u$ is a generalized solution of \cref{eq122}, $\phi\in W^{1,2}(\Omega)$ and $u-\phi \in W^{1,2}_{0}(\Omega)$.

With a modification of the argument of Theorem 8.1 of \cite{Gt}, we first outline a routine proof of the maximum principle, which will be used in the course of our proof.  

\begin{theorem}[Weak Maximum principle]\label{Maximum}
    Consider the operator $L_1$ in \cref{imp22}, and consider $u\in W^{1,2}(\Omega)$ so that $L_1 u\geq 0$ ($L_1 u\leq 0$). Then,
    \begin{align}
      \sup\limits_{\Omega}  u\leq \sup\limits_{\partial\Omega} u^{+}, (\inf\limits_{\Omega}  u\geq \inf\limits_{\partial\Omega} u^{+})
  \end{align}
\end{theorem}
\begin{proof}[Outline of the proof:]
    The proof follows almost exactly as in Theorem 8.1 of \cite{Gt}; when $\B=0$ the proof is the same as the corresponding case in \cite{Gt}. In the case that $\B\neq 0$, note that it remains to consider the case where $\sup\limits_{\Omega} u$ is only attained in the interior of the domain $\Omega$, say at a point $x_0$ and thus considering $l=\sup\limits_{\partial\Omega} u^{+}$ and in the case that $l=\sup\limits_{\partial\Omega} u^{+}< \sup\limits_{\Omega} u$, we consider a $l\leq k<\sup\limits_{\Omega} u$ with $k\uparrow \sup\limits_{\Omega} u$. In that case the proof is completed by contradiction as in \cite{Gt} , by considering the compactly supported functions $(u-k)^{+}\in C^{1}_{0}(\Omega)$ , with a uniform upper bound on the magnitude of the $\B$ term in some ball $B(x_0,r)\subset \Omega$, where without loss of generality $r\leq \delta(x_0)/2$ . We omit the details.
\end{proof}
We immediately get from \cref{Maximum}, the following corollary, as Corollary 8.2 of \cite{Gt}.
\begin{corollary}
    Let $u\in W^{1,2}_{0}(\Omega)$ satisfy $L_1 u=0 $ in $\Omega$. Then $u=0$ in $\Omega$.
\end{corollary}
We note that \cref{nonpositivee} , guarantees, with the use of the argument of Theorem 8.1 of \cite{Gt} as above, that the solution to the adjoint equation when it exists, satisfies the maximum principle as well. This will be used in the proof of \cref{thm3} as well.

    
 

As in Section 4 of Chapter III of \cite{Singdr}, we assume that the Dirichlet problem is solvable for the operator $L_1$, and that the Green's function is also well defined when taking limiting functions for the Dirac delta distribution at the pole of the Green's function. We would conclude in the limit that the Green's function is positive.

For further consideration, for any point $y\in \Omega$, call $\Omega_\rho=\Omega\cap B_{\rho}(y)$ where $B_{\rho}(y)$ is the ball of radius $\rho$ centered at $y$. Also define,
\begin{align}\label{eq177}
    f_{\rho}(x,y)=|B_{\rho}(y)|^{-1}\textbf{1}_{\Omega_{\rho}}(x), \  x\in \Omega.
\end{align}

\section{Lower bounds on the Green's function.}

The lower bound follows by an argument similar to that used in \cite{GrWi,Pat24}. \footnote{This was also outlined in the argument of Lemma 4.3 in Section III in \cite{Singdr}.}

We prove this here.

\begin{proof}[Proof of \cref{thm1}]

    The proof essentially follows by extending the argument of the proof of Eq.(1.9) of \cite{GrWi}. Take $r:=|z-y|$. Consider a smooth cut-off function $\eta$ which is $1$ on $B_r(y)\setminus B_{r/2}(y)$ and zero outside $B_{3r/2}(y)\setminus B_{r/4}(y)$, and further $0\leq \eta\leq 1$ and $|\nabla \eta| \leq \frac{K}{r}$.

    Henceforth, we use the Einstein summation convention, where the summation sign is implied.

    Given the domain $\Omega$, for any admissible test function $\phi$, the Green's function satisfies the following adjoint equation;
    \begin{equation}\label{eqimpp}
        \int_{\Omega} \Big((\nabla \phi) \cdot\nabla \G(y,x) -\G(y,x)\B\cdot(\nabla \phi) \Big)dx =\phi(y)
    \end{equation}
    Here $\A_{ij}^{T}$ denotes the transpose of the matrix $\A_{ij}$. 
    
We note that for $\Omega$, with $D=\text{diam}(\Omega)<\infty$, we have for the fixed $\beta>0$, that,
\begin{align}
    \Big(\frac{1}{\delta(x)}\Big)^{1-\beta} \leq \frac{D^{\beta}}{\delta(x)}.
\end{align}

    Consider the test function $\phi= G(\cdot,y)\eta$, the bound on the drift term $\B$, the Cauchy inequality with $\epsilon$'s, and the bounds on the cut-off function $\eta$ introduced above, that,
   \begin{multline}
        \int\limits_{r/2<|x-y|<r} |\nabla \G(y,x)|^{2} dx \leq \Big(K_1 \frac{1}{r^{2}}\cdot \int\limits_{r/4<|x-y|<3r/2} \G(y,x)^{2}dx\Big) +\Big(K_2 \frac{1}{r\delta(y)}\cdot\int\limits_{r/4<|x-y|<3r/2} \G(y,x)^{2} dx\Big) \\ + \Big(\frac{K_2}{\delta(y)}\cdot \int\limits_{r/4<|x-y|<3r/2} \G(y,x)|\nabla \G(y,x)|dx\Big)        
    \end{multline}

    Noting that we have $r\leq \frac{1}{2}\delta(y)$, using the Cauchy inequality with $\eps$'s again for the last term on the right, and finally adjusting the terms, we get
       \begin{align}\label{eq6}
        \int\limits_{r/2<|x-y|<r} |\nabla \G(y,x)|^{2} dx \leq \tilde{K} \frac{1}{r^{2}}\cdot \Big(\int\limits_{r/4<|x-y|<3r/2} \G(y,x)^{2}dx\Big)\leq \tilde{K}r^{n-2}\big(  \sup\limits_{r/4\leq |x-y|\leq 3r/2} \G(y,x)^{2} \big) .
    \end{align}

    Again as in \cite{GrWi}, choose a similar cut-off function $\phi$ that is 1 on $B_{r/2}(y)$ and zero outside $B_{r}(y)$, and using it as the test function we get,
    \begin{multline}
     1= \int\limits_{r/2\leq |x-y|\leq r} (\A_{ij}\partial_{i}\G(y,x) \partial_{j} \phi + \G(y,x)\B_i \partial_i \phi) \Big) dx)\leq M \frac{K}{r} \int\limits_{r/2\leq |x-y|\leq r} |\nabla \G(y,x)(\cdot,y)|dx \\+\Big(\frac{MK}{r\delta(y)}\cdot\int\limits_{r/2\leq |x-y|\leq r} |\G(y,x)| dx)
    \end{multline}


    Using the identity of \cref{eq6}, and Cauchy's inequality for the first term on the right, along with a trivial volume bound, and finally Harnack's inequality,
\begin{align}
     1\leq K r^{n-2} \sup\limits_{r/4\leq |x-y|\leq 3r/2} |\G(y,x)|
     \leq K |z-y|^{n-2}|\G(y,z)|.
\end{align}
     \end{proof}



\section{Upper bounds on the Green's function.}
Here we prove \cref{thm2}.

We will prove this upper bound on the Green's function $\G(x,0)$ for any $x\in \Omega$ and with $|x|\leq \frac{1}{2}$. In doing so, we initially adopt the argument of Proposition 5.10 of \cite{KimSa} and introduce a modified Lorentz norm for our purpose. 


\begin{proof}[Proof of \cref{thm2}] For any $\rho$, consider as in \cite{GrWi} and in the case of Proposition 5.10 of \cite{KimSa} the test function given by 
\begin{align}
    \phi(x)=\Big(\frac{1}{s}-\frac{1}{\G_{\rho}(x,0)}\Big)^{+}, \ \ \text{and} \ \ \Omega_{s}:=\{x\in \Omega: \G_{\rho}(x,0)\geq s \}.
\end{align}

Here we use the standard notation that $h(x)^{+}:=\text{max}(h(x),0)$. 

Note that we have by definition,
$$LG_\rho(x,0)= f_\rho (x,0).$$

Using $\phi$ as a test function, we get using the above and an integration by parts, that,
\begin{align}\label{eq12}
    \int_{\Omega_{s}} \nabla\G_{\rho}(x,0)\cdot \nabla \phi dx\leq \frac{1}{s}+\int_{\Omega_{s}}\B\cdot (\nabla \G_{\rho}(x,0)) \phi dx.
\end{align}

We have that $\nabla \phi=\nabla \G_{\rho}(\cdot,0)/\G_{\rho}(\cdot,0)^{2}$ where the derivative is understood to be with respect to the $x$-variable.

Using \cref{nonpositivee2}, using a standard integral by parts, we get,
\begin{align}
    \int_{\Omega_{s}}\nabla\G_{\rho}(x,0)\cdot \frac{\nabla \G_{\rho}(x,0)}{\G_{\rho}(x,0)^{2}}dx \leq \frac{1}{s}+ \int_{\Omega_{s}}\B\cdot \Big(\frac{ \nabla \G_{\rho}(x,0)}{\G_{\rho}(x,0)}\Big)dx .
\end{align}

For the given fixed $y$, set $w_{s,\rho} (x)=(\ln(\G_{\rho}(x,0)/s))^{+}$. In this case, $\nabla w_{s,\rho}(x)= \nabla \G_{\rho}(x,y)/\G_{\rho}(x,y)$ in the set $\Omega_{s}$. Then using the ellipticity condition on $\A$ and the triangle inequality and Cauchy inequality, we have 
\begin{align}
    \lambda\int_{\Omega_s} |\nabla w_{s,\rho}|^{2} dx\leq \frac{1}{s}+\int_{\Omega_{s}}|\B||\nabla w_{s,\rho}|dx\leq \frac{1}{s} +\frac{\lambda}{2}\int_{\Omega_s} |\nabla w_{s,\rho}|^{2} +\frac{1}{2\lambda}\int_{\Omega_s}|\B|^{2}dx,
\end{align}
and thus we have 
\begin{align}
    \frac{\lambda}{2} \int_{\Omega_s} |\nabla w_{s,\rho}|^{2} dx\leq \frac{1}{s}+\frac{1}{2\lambda}\int_{\Omega_s}|\B|^{2}dx.
\end{align}

Next we use the Sobolev inequality, and Holder's inequality on the last term on the right, and noting $\Omega_{2s}\subset \Omega_{s}$ we have,  
\begin{align}
    \frac{\lambda}{2}(\ln 2)^{2}|\Omega_{2s}|^{\frac{n-2}{n}}\leq \frac{\lambda}{2}\Big( \int_{\Omega_s} |w_{s,\rho}|^{\frac{2n}{n-2}} dx\Big)^{\frac{n-2}{n}} \leq C_{n}^{2}\frac{\lambda}{2}\Big( \int_{\Omega_s} |\nabla w_{s,\rho}|^{2} dx\Big) 
    \\ \leq \frac{C_{n}^{2}}{s}+\frac{C_{n}^{2}}{2\lambda}\Big(\int_{\Omega_s} |\B|^{n} dx\Big)^{\frac{2}{n}} |\Omega_s|^{\frac{n-2}{n}}.
\end{align}

Write $f(s)=|\Omega_{s}|^{\frac{n-2}{n}}$.
Thus, we have constants $K_1, K_2>0$ so that 
\begin{align}\label{eq16}
    f(2s)\leq \frac{K_1}{s} +K_2\Big(\int_{\Omega_s} |\B|^{n} dx\Big)^{\frac{2}{n}} f(s).
\end{align}
We make a digression to state that all the bounds we obtain in the end of this section, will be independent of $\rho$, and without loss of generality, we can henceforth work with the Green's function $\G(\cdot,0)$ in place of $\G_{\rho}(\cdot,0)$. 

Now consider the following two cases:
\begin{itemize}
    \item[(i)] There exists a large enough absolute constant $C$ so that for any configuration of the drift in the ball $B(0,1)$, we have $f(2s)\geq \frac{1}{C}f(s)$ for all $ s\geq \widetilde{s}:=\max\limits_{x\in S(0,\frac{1}{2})}\G(x,0) $.
    \item[(ii)] The above fails, and so for all positive integers $N$ large enough, we have some drift configuration $|\B_N|$ for which we have $f(s_N)>N f(2s_N)$ for some $s_N\geq \widetilde{s_N}$.  Here we define, $\widetilde{s_N}:=\max\limits_{x\in S(0,\frac{1}{2})}\G_N(x,0) $, where $\G_N(x,0)$ is the Green's function corresponding to the operator with the drift $|\B_N|$. We deal with this case later, in a manner similar to the argument used in the later part of the argument for Case $(i)$.
\end{itemize}
\bigskip

\begin{enumerate}

\item[Case(i).] We deal with Case(i) above. Consider a fixed positive real number $K\leq 1/(2C)$. Now consider a real number $L\geq 2$ large enough so that upon defining $s^{*}:=\max\limits_{x\in S(0,1/L)}\G(x,0)$, for $s\geq s^{*}$, noting the bound $|\B|\lesssim \frac{M}{\delta_{\Omega}(y)}$ in $\Omega_{s}$, and using the maximum principle and noting that $\Omega_{s}\subset \overline{B_{0,L}}:=\overline{B(0,1/L)}$ for $s\geq s^{*}$, using a trivial volume bound on the integral on the right of \cref{eq16}, we get for all $s\geq s^{*}$,
\begin{multline}\label{eq171}
    f(2s)\leq \frac{K_1}{s} +K_2\Big(\int_{B(0,1/L)} |\B|^{n} dx\Big)^{\frac{2}{n}} f_{y_0}(s) \leq  \frac{K_1}{s} +K_2\Big(\int_{B(0,1/L)} \Big|\frac{M}{\delta_{\Omega}(y)}\Big|^{n} dx\Big)^{\frac{2}{n}} f_{y_0}(s) \\ \leq  \frac{K_1}{s} +K f_{y_0}(s). 
\end{multline}

We have required that $L$ be large enough so that $K_2 c(\frac{M}{L})^{2}\leq  K$ with the constant volume prefactor of $c$, so that the last inequality in \cref{eq171} holds.

Thus we have for all $s\geq s^{*}$,
\begin{align*}
     f(2s)-KCf(2s) \leq f(2s)-Kf(s)\leq \frac{K_1}{s}.
\end{align*}

Since we have $KC\leq 1/2$, we get,
\begin{align*}
   \frac{1}{2}f(2s) \leq (1-KC)f(2s)\leq \frac{K_1}{s},
\end{align*}

and so for all $\lambda\geq 2s^{*}$, we have 
\begin{align}\label{eq1888}
    f(\lambda)=|\{x:\G(x)\geq \lambda\}|^{\frac{n-2}{n}}\leq \frac{K_3}{\lambda}.
\end{align}

Consider any $p> 1$\footnote{Later we will choose $p=\frac{n}{n-2}$.}. We define a modified Lorentz-norm type functional for the Green's function, given by 
\begin{align}\label{ModifiedLorentz}
    ||\G(\cdot, 0)||^{*}_{p,\infty}:=\sup\limits_{\lambda\geq  2s^{*}} \lambda |\Omega_{\lambda}|^{\frac{1}{p}}=\sup\limits_{\lambda \geq  2s^{*}} \lambda|\{x:\G(x)\geq \lambda\}|^{\frac{1}{p}}.
\end{align}

Consider any $p>\alpha>1$. Modifying a standard argument found for example in the proof of Theorem 1.2.10 of \cite{Grafakos}, for any subset $E\subset \Omega_t$ with $t\geq 2s^{*}$, we get,
\begin{align}\label{eq20}
    \int\limits_{E} \G(x,0)^{\alpha}dx =\int_{0}^{\infty} |\{x:\G(x,0)\geq \lambda^{\frac{1}{\alpha}}\}\cap E|d\lambda
    \leq \int_{0}^{\infty} \min\Big( |E|,\{x:\G(x,0)\geq \lambda^{\frac{1}{\alpha}}\}\Big)d\lambda 
\end{align}

Consider the value $\lambda_{E}$, defined as the unique value for which $|E|=|\{x:\G(x,0)\geq \lambda_{E}^{\frac{1}{\alpha}}\}|$. In the last integral on the right of \cref{eq20}, we have 
\begin{align}\label{eq21}
   \int_{0}^{\infty} \min\Big( |E|,|\{x:\G(x,0)\geq \lambda^{\frac{1}{\alpha}}\}|\Big)d\lambda = \int_{0}^{\lambda_E} |E|d\lambda +\int_{\lambda_E}^{\infty}| \{ x:\G(x,0)\geq \lambda^{\frac{1}{\alpha}} \} |d\lambda
\end{align}

We consider the set $S_{B_{0,L}}$ of all measurable subsets $E\subset B(0,1/L)$ so that $\lambda_E\geq (2s^{*})^{\alpha}$ and $0\notin E$. From the above consideration and the monotonic behavior of this distribution function of $\G(x,0)$, the set $S_{B_{0,L}} $ consists of all subsets $E\subset B(0,1/L)$ so that $|E|=|\{x:\G(x,0)\geq \lambda_{E}^{\frac{1}{\alpha}}\}|\leq |\{x:\G(x,0)\geq 2s^{*}\}|$:
\begin{align}\label{defineset}
    S_{B_{0,L}}=\{E\subset B(0,1/L), 0\notin E: |E|\leq |\{x:\G(x,0)\geq 2s^{*}\}|\}.
\end{align}

In this case, for any $E\subset S_{B_{0,L}}$, for which $\lambda_{E}^{\frac{1}{\alpha}} \geq (2s^{*})$, we have
\begin{align}\label{eq22}
    |E|=|\{x:\G(x,0)\geq \lambda_{E}^{\frac{1}{\alpha}}\}|\leq \frac{(||\G(\cdot, 0)||^{*}_{p,\infty})^{p}}{(\lambda_{E})^{\frac{p}{\alpha}}},
\end{align}
by the definition of \cref{ModifiedLorentz}.

Since the function $\lambda\to (||\G(\cdot, 0)||^{*}_{p,\infty})^{p}/\lambda^{p/\alpha}$ is decreasing, by \cref{eq22}, the value $\lambda^{'}_{E}$ that satisfies 
\begin{align}\label{eq23}
    |E|=\frac{(||\G(\cdot, 0)||^{*}_{p,\infty})^{p}}{(\lambda^{'}_{E})^{\frac{p}{\alpha}}}
\end{align}
is such that $(\lambda^{'}_{E})^{\frac{1}{\alpha}}\geq (\lambda_{E})^{\frac{1}{\alpha}}\geq (2s^{*})$.

Note that 
\begin{align}
    (\lambda^{'}_{E})=\frac{||\G(\cdot,0)||^{*}_{p,\infty})^{\alpha}}{|E|^{\frac{\alpha}{p}}}.
\end{align}

Thus for each $E\subset 
 S_{B_{0,L}}$ we can write 
\begin{multline}
    \int_{0}^{\infty} \min\Big( |E|,|\{x:\G(x,0)\geq \lambda^{\frac{1}{\alpha}}\}|\Big)d\lambda \leq \int_{0}^{\infty} \min\Big( |E|,\frac{(||\G(\cdot, 0)||^{*}_{p,\infty})^{p}}{\lambda^{\frac{p}{\alpha}}}\Big)d\lambda \\ =   \int_{0}^{\lambda^{'}_E} |E|d\lambda +\int_{\lambda^{'}_E}^{\infty} \frac{(||\G(\cdot, 0)||^{*}_{p,\infty})^{p}}{\lambda^{\frac{p}{\alpha}}}d\lambda.
\end{multline}

Now from the value of $\lambda^{'}_{E}$ given in \cref{eq23}, we get upon performing the integrations above \footnote{Such an expression appears in the proof of Theorem 1.2.10 of \cite{Grafakos}}, 
\begin{multline}\label{eq26}
    \int\limits_{E} \G(x,0)^{\alpha}dx =\int_{0}^{\infty} |\{x:\G(x,0)\geq \lambda^{\frac{1}{\alpha}}\}\cap E|d\lambda \leq\int\limits_{0}^{\infty}\min\Big( |E|,|\{x:\G(x,0)\geq \lambda^\frac{1}{\alpha}{}\}|\Big)d\lambda\\ \leq \frac{p}{p-\alpha}|E|^{(1-\frac{\alpha}{p})}(||\G(\cdot, 0)||^{*}_{p,\infty})^{\alpha}.
\end{multline}

From \cref{eq1888,ModifiedLorentz}, we have 
\begin{align}
    ||\G(\cdot, 0)||^{*}_{\frac{n}{n-2},\infty}\leq K_3.
\end{align}

We thus have for all $E\subset S_{B_0,L}$, with $1<\alpha<p=n/(n-2)$,
\begin{align}
    \Big(\frac{1}{|E|}\int\limits_{E} \G(x,0)^{\alpha}dx\Big)^{\frac{1}{\alpha}}\leq \frac{n}{n-\alpha(n-2)}|E|^{\frac{2-n}{n}}K_3.
\end{align}


For the special case of $E=B(y,R)$ being a ball of some radius $R$, Moser's inequality, (see for example Theorem 8.17 of \cite{Gt}, we get that \footnote{It is to use this inequality with the drift term that we need an exponent on $\G$ of $\alpha>1$. }  
\begin{align}\label{eq31}
    \sup\limits_{x\in E/2} \G(,0)\leq C' \Big(\frac{1}{|E|}\int\limits_{E} \G(x,0)^{\alpha}dx\Big)^{\frac{1}{\alpha}}\leq \frac{C'n}{n-\alpha(n-2)}|E|^{\frac{2-n}{n}}K_3,
\end{align}
 for some constant $ C'(n,\lambda,M,R)$ and in our domains of the form $B(0, 1/L)$ in consideration, along with the drift term,  $R$ is bounded from above. Here $\frac{E}{2}$ denotes the ball with half the radius of $E$ but with the same center.

Now in particular in $S_{B_{0,L}}$, we choose a ball $E$ with center $y\in B_{0,L}$, and with radius $r_y :=|y|/4$. Then from \cref{eq31} we get in particular,
\begin{align}\label{eq32}
    \G(y,0)\leq \frac{C_1}{|y|^{n-2}}.
\end{align}

Define the subset of balls of the form $S'_{B_{y_0},L}=\{B(y,|y|/4)\subset S_{B_{0},L} \}$. Each ball in $S'_{B_{0},L}$ has volume $C_2(|y|)^{n}$ for some uniform constant $C_2$.

We note from \cref{defineset} that the volume of sets in $S_{B_{0,L}}$ are upper bounded by $V(0):=|\{x:\G(x,0)\geq 2s^{*}\}|$. Thus, there is an upper bound on the distance from the origin to the centers of the balls in $S'_{B_{0},L}$, depending on the value of $V(0)$. In particular, there is a value $r_{0}$ dependent on $V(0)$ so that for all $z\in \mathbb{B}_{0}:=B(0,r_{0})$, the balls $B(z,r_z)=B(z,|z|/4)\subset S'_{B_{0},L}$ and also $B(z,r_z)\subset B(0,1/L)$. Here we have defined $r_z =|z|/4$.

\begin{figure}[h]
\centering
\includegraphics[width=0.5\textwidth]{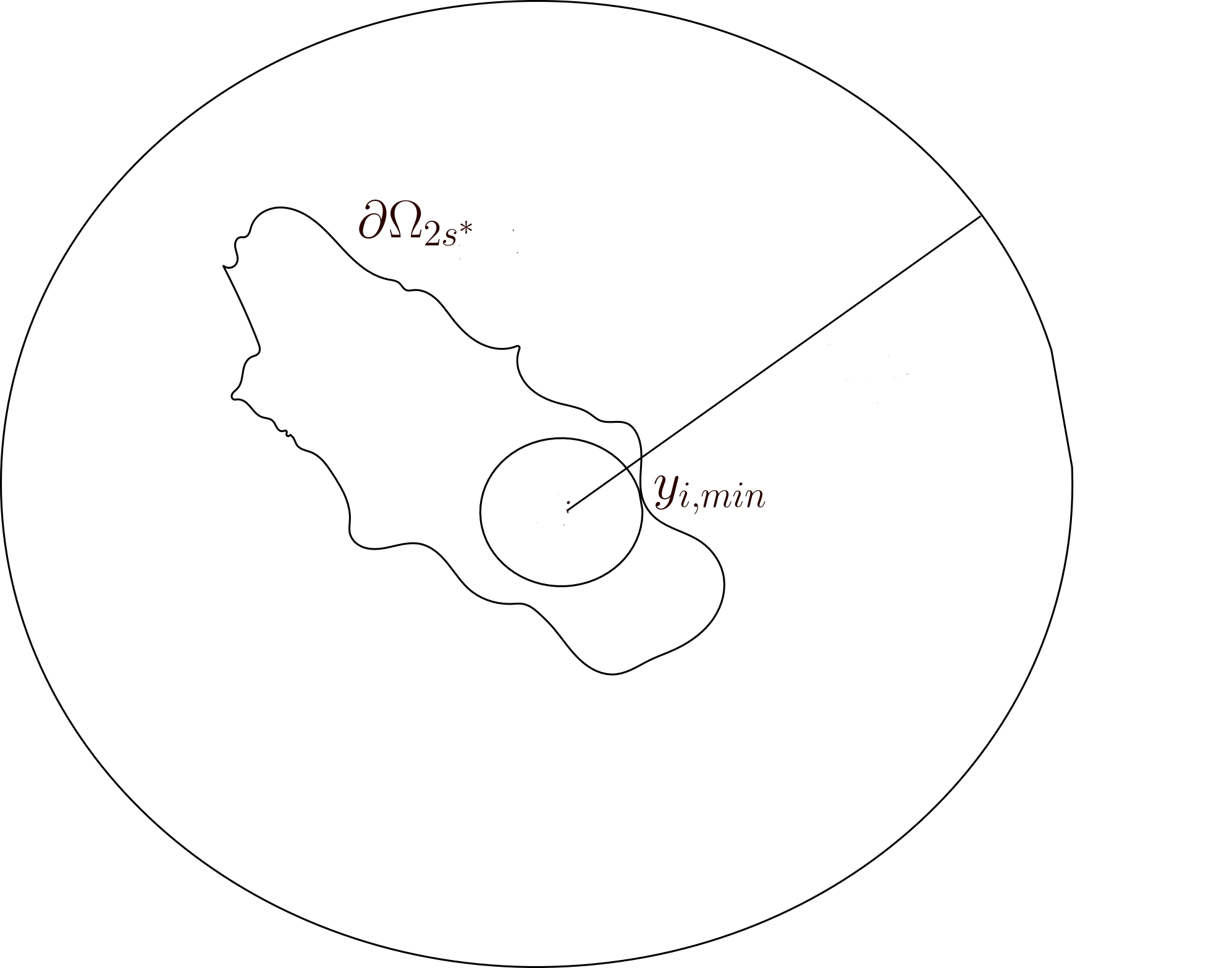}
\caption{The boundary of the set $\Omega_{2s^{*}}$ , the sphere of radius $1/L$ centered at $0$, the point on the level set $\partial\Omega_{2s^{*}}$ at the closest distance from $0$, are shown in the figure.}
\label{fig:inksone3}
\end{figure}

\bigskip

We next show that there exists some constant $\eta>0$ so that for any configuration of the drift, we have
\begin{align}\label{imp}
    \frac{V(0)}{|B(0,1/L)|}=\frac{|\{x:\G(x,0)\geq 2s^{*}\}|}{|B(0,1/L)|}\geq \eta.
\end{align}

This clearly means that for any configuration of the drift, we have
\begin{align}
    \frac{\mathbb{B}_{0}}{|B(0,1/L)|}\geq \eta.
\end{align}

Thus the estimate of \cref{eq32} clearly holds for all $z\in \mathbb{B}_0$, and with a routine Harnack inequality, one extends the same inequality with a modified constant to the entire ball $ B(0,1/L)$.

\bigskip

The proof of \cref{imp} and the proof of Case(ii) are similar in nature. In particular, in proving \cref{imp} and thus Case(i), henceforth we no longer have to use the fact that there is a constant $C$ so that $f(2s)\geq \frac{1}{C}f(s)$  for all $s\geq \tilde{s}$ ,and the arguments henceforth can also be employed for Case(ii).  We argue by contradiction; that if there does not exist the claimed $\eta>0$, then there exists a sequence of drifts $|\B_i|$ satisfying \cref{Laplacian} and \cref{nonpositivee} for which the ratio of \cref{imp} goes to zero. In particular, if we write the Green's function corresponding to the drift $|\B_i|$ as $\G_i(\cdot,0)$, and $s^{*}_i:=\max\limits_{x\in S(0,1/L)}\G_i(x,0)$ , we have as $i\to \infty$, 
\begin{align}
    \frac{|\{x:\G_i(x,0)\geq 2s^{*}_i \}|}{|B(0,1/L)|}\to 0.
\end{align}

In particular, the point $y_{i,\text{min}}$ at the minimum distance from the origin to the level set of $|\{x:\G_i(x,0)= 2s^{*}_i \}|$ is so that $y_{i,\text{min}}\to 0$ as $i\to \infty$.

We will use the result of \cref{thm1} for our argument as well. 

\bigskip

We first establish the following lemma which will be used later on. Hereafter we write $\G(x,0)=\G(x)$.
\begin{lemma}\label{lemma}There is a uniform constant $M_0$ so that for any configuration of the drift  that satisfies \cref{Laplacian} and \cref{nonpositivee}, we have a sequence $b_{k}|_{k=1}^\infty$ of maximum points, and a sequence $a_{k}|_{k=1}^\infty$ of minimum points, so  that, 
\begin{align}\label{eq89}
    \Big|\frac{\partial\G}{\partial r}\Big|_{b_{k}} b_{k}^{n-1}< M_0 \Big|\frac{\partial\G}{\partial r}\Big|_{a_{k}} a_{k}^{n-1},
\end{align}
and $a_{k}\to 0$, and $b_{k}\to 0$ as $k\to \infty$. In the above equation, it is understood that the pole of the Green's function is at $0$.\end{lemma}
We need to be careful below about defining the sequence of $a_{k}|_{k=1}^{\infty}$, since we do not have a-priori information about the structure of the level sets of the Green's function. 

\begin{proof}[Proof of Lemma 5]Assume to the contrary, that for any arbitrarily large $M(\geq 1)$, there exists some  $i(M)$ and drift $|\B_i|$ satisfying \cref{Laplacian} and \cref{nonpositivee} so that, we have, eventually for any sequence of maximum points $b_{i,k}|_{k=1}^{\infty}$ and any sequence of minimum points $a_{i,k}|_{k=1}^{\infty}$, each sequence converging to $0$, 
\begin{align}
    I=\inf\limits_{k}\Big(\Big|\frac{\partial\G_i}{\partial r}\Big|_{b_{i,k}} b_{i,k}^{n-1}\Big) \geq M \sup\limits_{k}\Big(\Big|\frac{\partial\G_i}{\partial r}\Big|_{a_{i,k}} a_{i,k}^{n-1}\Big):=M\cdot S,
\end{align}
Fix some radius $r_0>0$. Without loss of generality, choose the maximum point $b_{i,k_0}=r_0$ and the minimum point $a_{i,k_0}$ to belong to the level set $\G(b_{i,k_0})$, for some $k_0 \geq 1$.  We suppress the dependence of $i(M)$ on $M$ and the dependence of $\G_i$ on $i$, below. 

Initially we choose the increment $h_0$ so that, 
\begin{align}\label{eq91}
    h_{0} \cdot \sup\limits_{\B(0, b_{i,k_0})\setminus B(0,\frac{1}{2} a_{i,k_0 })}|\G^{''}_i (x)|\ll  \frac{S}{(b_{i,k_0 })^{n-1}}\leq \frac{I}{(b_{i,k_0 })^{n-1}}, \ h_0\ll \frac{1}{2}a_{i,k_0 }.
\end{align}
In this case, define the two sequences in the following way: first consider the point $a_{i,(k_0+1)}$ which is the minimum point of the level set $\{x :\G(x)=\G(a_{i,k_0 }-h_0)\}$. By definition, we have, $|(a_{i,k_0 }-h_0)|\geq |a_{i,k_0 +1 }|$.  We continue this process, till we reach some $t_1\geq 1$ so that $a_{i,k_0 +t_1 -1}\geq \frac{1}{2}a_{i,k_0}$ with possibly $h'_0 =|a_{i, k_0 +t_1 -1}  -\frac{1}{2}a_{i,k_0}|\leq h_0$, while at this last stage we define $a_{i, k_0 +t_1}$ as the minimum point of the level set $\G(a_{i,k_0 +t_1 -1}-h'_0)=\G(\frac{1}{2}a_{i,k_0 +t_1})$.

We define the sequence $a_{i,k_0 +t_m}|_{m=1}^{\infty}$, for each $m\geq 1$ in the same manner as above. We take,
\begin{align}\label{eq91'}
    h_{m} \cdot \sup\limits_{B(0, b_{i,k_0})\setminus B(0,\frac{1}{2} a_{i,k_0 +t_m})}|\G^{''}_i(x)|\ll  \frac{S}{(b_{i,k_0 })^{n-1}}\leq \frac{I}{(b_{i,k_0 })^{n-1}}, \ h_m\ll \frac{1}{2}a_{i,k_0 +t_{m} }.
\end{align}
Here the implied constants are allowed to depend on $M,S$ as well. Here $t_{m+1}$, for $ m\geq 1$ is the greatest integer so that $a_{i,k_0 +t_{m+1} -1}\geq \frac{1}{2}a_{i,k_0 + t_{m}}$. For all the values $\{a_{i,k_0 +t_m},\dots,a_{i,k_0 +t_{m+1} -1}\}$, the decrement at each step is taken as $h_m$, with $a_{i,k_0 +t_{m}+j}$, for $1\leq j\leq t_{m+1}-t_{m}$, being the minimum point of the level set of $\G(a_{i,k_0 + t_{m}+ j-1}-h_m)$.

It is clear that by this process, the sequence $a_{i,k}|_{k=1}^{\infty}$ so defined converges to $y_i$. We continue this iteration till we reach some value $t_{m_0}=N$, with $a_{i, k_0 +N}=\eps$ so that $\G(a_{i,k_{0}+N})=\G(\eps)\geq K(M,\lambda)\Big(\frac{1}{\eps}\Big)^{n-2}\gg \G(b_{i,k_0})= \G(a_{i,k_0})$, where we have invoked the result of Theorem 1, proved in Section 3.

Starting with $b_{i,k_0}$, we define $b_{i,k+t}$ as the point on the sphere of radius $|b_{i,k+t-1}-h_{m_0})|$ where the maximum value of the Green's function is attained. By definition, we have that $\G(b_{i,k+t})\geq \G(b_{i,k+t-1}-h_{m_0}) $.

Then, we have, for each  $k\geq 1$ in case of the sequence of $b_{i,k}$'s, and for each $1\leq j\leq t_{m+1}-t_{m}-1$, for the case of the sequence of the $a_{i,k}$'s, that,
\begin{multline}\label{eq921}
    \G(b_{i,k}-h_{m_0})- \G(b_{i,k}) =\Big|\frac{\partial\G}{\partial r}\Big|_{b_{i,k}} \cdot h_{m_0} +O_{m_0}(h_{m_0}^2), \\  
    \G(a_{i,k_0 + t_{m}+ j-1}-h_m)- \G(a_{i,k_0 + t_{m}+ j-1})=\Big|\frac{\partial\G}{\partial r}\Big|_{a_{i,k_0 + t_{m}+ j-1}} \cdot h_m + O_m(h_{m}^2),
\end{multline}

 Here the implied constants depend on the bounds on the second derivative of $\G$ in $\Omega\setminus B(y_i,a_{i,k_0 + t_{m}+j-1})$ for each $j$ considered. 

So we have,
\begin{align}\label{eq98}
     \Big(\sum\limits_{q=0}^{N}\frac{h_q S}{(a_{i,(k_0 +q)} )^{n-1}}+ O_q(h_{q}^2)\Big)\geq\G(a_{i,k_0 +N})- \G(a_{i,k_0})= \G(\eps) -\G(a_{i,k_0})\approx \G(\eps )\gg \G(a_{i,k_0}).
\end{align}
Here we have $h_q=h_m$ whenever $ k_0 +t_{m}\leq q\leq k_0 +t_{m+1} $ . 

Using \cref{eq91,eq91'}, we find a slightly altered constant $S'$, so that we have, 
\begin{align}
    \Big(\sum\limits_{q=0}^{N}\frac{h_q S'}{(a_{i,(k_0 +q)} )^{n-1}}\Big)\geq \G(\eps) \gg \G(a_{i,k_0}).
\end{align}
In this case, we add up the equality in \cref{eq921} above. First we get,
\begin{align}
    \G(b_{i,k_0+N_1}) -\G(b_{i,k_0})\geq  \sum_{k}\Big(\G(b_{i,k}-h_{m_0})- \G(b_{i,k}) \Big)\geq c\Big(\sum\limits_{q=1}^{N_1}\frac{Ih_{m_0}}{(b_{i,(k_0 +q)} )^{n-1}} +O_{m_0} (h_{m_0}^2)\Big)    ,
\end{align}
where $b_{i,k_{0}+N_1}$ lies in the domain $\Omega\setminus B(y_i,\eps)$, with the property that, 
\begin{align}
    |b_{i,k_0 +N_1} - \eps|\leq h_{m_0}.
\end{align}

We terminate the sequence of $b_{i,j}$'s at $j=N_1$. Clearly $N_1\geq N$. 

In this case, from \cref{eq91'}, we choose the implied constants so that $h_{m_0}$ is small enough, so at each step we have, 
\begin{align}
     \frac{I}{(b_{i,k_0 +q})^{n-1}}\gg h_{m_0},
\end{align}
for all $b_{i,k_0 +q}\in \Omega\setminus B(y_i, \eps)$, and adding this up $N_1$ many times, we get with a slightly altered constant that, 
\begin{align}\label{eq94}
 \G(b_{i,k_0+N_1})\geq  \G(b_{i,k_0+N_1}) -\G(b_{i,k_0})\geq cI'\Big(\sum\limits_{m=0}^{N_1}\frac{h_{m_0}}{(b_{i,(k_0 +m)} )^{n-1}}\Big)\geq cMS'\Big(\sum\limits_{m=0}^{N_1}\frac{h_{m_0}}{(b_{i,(k_0 +m)} )^{n-1}}\Big) ,
\end{align}
for a slightly altered constant $I'$.\footnote{In fact, one can also work with this same $h_{m_0}$ for the sequences of $a_{i,k}|_{k=k_0}^{\infty}$ in going from $a_{i,k_0}$ to $a_{i,k_0 +N}=\eps$, which would slightly alter the argument presented above.}

Further we also have, 
\begin{align}
     \G(a_{i,k_0+N}) -\G(a_{i,k_0})\leq c\Big(\sum\limits_{q=0}^{N}\frac{S h_q}{(a_{i,(k_0 +q)} )^{n-1}}   + O_q( h_{q}^2)\Big),
\end{align}
where as before, we have $h_q=h_m$ whenever $ k_0 +t_{m}\leq q\leq k_0 +t_{m+1} $

 Again, we have, by a slightly altered constant $S'$, that,
\begin{align}\label{eq96}
    \G(a_{i,k_0 +N}) \approx  \G(a_{i,k_0+N}) -\G(a_{i,k_0})\leq cS'\Big(\sum\limits_{q=0}^{N}\frac{h_q}{(a_{i,(k_0 +q)} )^{n-1}}\Big),
\end{align}
and without loss of generality, we can take $I'\geq M S'$ as above. 

Thus for $M$ large enough, from \cref{eq94,eq96,eq98} using a basic integral test, comparing the values of  $\G(b_{i,k_0+N_1}), \G(a_{i,k_0+N})$, we have a violation of Harnack's inequality, on an annular region of width $h$ centered on the sphere $S(b_{i,k_0 +N_1})$. This is because of \cref{eq98}, and the right hand side of \cref{eq94} is then arbitrarily large compared to the right side of \cref{eq96}.

Thus \cref{eq89} is established. \end{proof}

We continue with the proof of \cref{thm2}, by contradiction. The sets $\Omega_{2s^{*}}=\{x:\G(x,0)\geq 2s^{*}\}\subset B(0,1/L)$, with 
\begin{align}\label{eq45}
    \frac{|\{x:\G_i(x,0)\geq 2s^{*}_{i}\}|}{|B(0,1/L)|}\to 0,
\end{align}
as $i\to \infty$.

Consider $i$ large enough to be determined later, and the point $y_{i,\text{min}}$ on the boundary $\partial \Omega_{2s^{*}_{i}}$ that is the minimum distance $\delta_{\text{min},y_i}$ from the origin. As a consequence of \cref{eq45}, for the uniform $L$, we have $\delta_{\text{min},y_i}\to 0$ as $i\to \infty$. 

In this case, the ball $B(0,\delta_{\text{min},y_i})\subset \Omega_{2s^{*}_{i}}$. By the Harnack inequality, there is a constant $\beta\gg 1$ independent of $r$ so that for all points on the sphere $S(0,r)$, the values of the Green's function with the pole at origin are comparable and we will have
\begin{align}
\sup\limits_{S(0,r)}\G_i(\cdot,0)\leq \beta \inf\limits_{S(0,r)}\G_i(\cdot,0).
\end{align} 
We thus have by definition, $\G_i(y_{i,\text{min}},0)\approx 2s^{*}_{y_i}$. 

Consider the sphere $S(0,y_{i,\text{min}})$ of radius $\delta_{\text{min}}$. In this case, we have $S(0,y_{i,\text{min}})\subset \Omega$ with a point of tangency of this sphere $S(0,y_{i,\text{min}})$ with $\partial \Omega_{2s^{*}_i}$ at the point $y_{i,\text{min}}$. At  $y_{i,\text{min}}$ the gradient $\nabla \G_i(\cdot,0)$ is inward normal to both $S(0,y_{i,\text{min}}),\partial \Omega_{2s^{*}_i}$\footnote{If there is any nonzero component of the gradient in the tangential direction, then to first order, one would find points $x_1,x_2 \in S(y_i,y_{i,\text{min}})$ on the interior sphere so that $\G(x_1,y_i)<\G(y_{i,\text{min}},y_i)<\G(x_2, y_i)$ which contradicts the fact that for all $x\in  \Omega_{2s^{*}_{y_i},y_i}$, by the maximum principle $\G(y_{i,\text{min}},y_i)\leq \G(x, y_i)$. }. Further, we must have $\frac{\partial^{2} \G_i(\cdot,0)}{\partial \theta_{k}^{2}}|_{y_{i,\text{min}}}\geq 0$ for any of the $(n-1)$ angular coordinates $\theta_k$ , otherwise up to second order,  noting from above that $\frac{\partial \G_i(\cdot,0)}{\partial \theta_k}|_{y_{i,\text{min}}}=0$ for all angular coordinates $\theta_i$, and that $\G_i(\cdot,0)$ is locally in $C^3$, we will find points on the interior sphere $x\in S(0,y_{i,\text{min}})$ with $\G_i(x,0)<\G_i(y_{i,\text{min}},0)$ that again contradicts the maximum principle.

Again, in this context, we suppress the dependence of $\G_i$ on $i$ below.

From the expression of the Laplacian locally in polar coordinates, up to a rotation of the coordinate axes, and noting that all the angular first derivatives are zero at the point $y_{i,\text{min}}$, and all the second derivatives with respect to the angular coordinates are positive, we will have
\begin{multline}\label{diffineq}
    \frac{1}{r^{n-1}}\frac{\partial}{\partial r}\big( r^{n-1}\frac{\partial \G(\cdot,0)}{\partial r} \big) +\B\cdot \hat{r} \frac{\partial \G(,0)}{\partial r}  = \frac{\partial^{2}\G(\cdot,0)}{\partial r^{2}} +\frac{(n-1)}{r}\frac{\partial \G(\cdot,0)}{\partial r}+ \B\cdot \hat{r} \frac{\partial \G(,0)}{\partial r} \leq 0 
\end{multline}

On the sphere $S(0,y_{i,\text{min}})$, the above inequality is valid at the point $y_{i,\text{min}}$.

For some $b_{i,k}$ to be determined in a later step, we begin the iteration with $r_0= a_{i,k}$ which is the minimum distance to the level set $\G_i(\cdot,0)=M_{b_{i,k}}$ where $M_{b_{i,k}}:=\sup\limits_{S(0, b_{i,k})}\G_i(\cdot,0)$.

We call $K_{i,k}:=\overline{B(0,1/L)}\setminus B(0,a_{i,k})$. When we write $\G(r)$, it will be understood from context that we mean $\G_i(r,0)$ . In particular, from the statement of \cref{lemma} we note that,
\begin{align}
    \Big|\frac{\partial\G_i}{\partial r}\Big|_{a_i,k}a_{i,k}^{n-1}=G_{i,k}>0.
\end{align}
Further, we write,
\begin{align}
    \sup\limits_{K_{i,k}} G_{i}''=H_{i,k}.
\end{align}

Now we choose the increment $h$ small enough so that,
\begin{align}\label{eq644}
     \Big|\frac{\partial\G_i}{\partial r}\Big|_{a_i,k}a_{i,k}^{n-1} L^{n-1}  \gg h \sup\limits_{K_{i,k}} G_{i}''=h H_{i,k}.
\end{align}

\footnote{Note that by the subsequent iteration, we will have for each $r\leq \frac{1}{L}$, \begin{align}
    \Big|\frac{\partial \G}{\partial r}\Big|_{r}\Big| \geq C\Big|\frac{\partial \G}{\partial r}\Big|_{a_{i,k}}\Big|\Big(\frac{a_{i,k}^{n-1}}{r^{n-1}}\Big),
\end{align} }
First we use the Taylor theorem to second order to get, for any $j\geq 0$, 
\begin{align}\label{eq655}
    \G(r_j +h)=\G(r_j) +h\G'(r_j) +\frac{h^2}{2}\G''(r_{j}+\theta_j h),
\end{align}
for some $\theta_j \leq 1$.
We have that for each $k$\footnote{We only need to consider $k\geq s^{*}_{y_i}.$}, using the Hopf lemma, the radial derivative $\frac{\partial \G}{\partial r}\leq 0$ at the point of tangency of each level set $\partial\Omega_{k}$ with this interior sphere. This holds true for each $j\geq 0$ with $ a_{i,k}\leq r_{j} \leq 1/L$.

At the first step for $j=0$, using \cref{eq644}, we get from \cref{eq655}, that, 
\begin{align}
    \G(r_0 +h) \leq G(r_0).
\end{align}
 We then define the length $h_0\leq h$ and the point $\overrightarrow{r_1}$ which is at the minimum distance from the origin to the level set of $\G=\G(r_0 +h)$ so that $\G(r_0 +h)-\G(r_0)=\G(r_{1})-\G(r_{1} -h_{0})$. where $r_{1}-h_0$ is on the radial line joining the origin to $r_{1}$. Note that $\G(r_{0}+h)=\G(r_{1})$ by definition and thus $\G(r_{0})=\G(r_{1}-h_0)$.

At the $j'$th step ($j\geq 1$), we consider $(r_j+h)\hat{r_j}$ where $\hat{r_{j}}$ is the unit vector in the direction $r_{j}$. Consider $\G((r_{j}+h)\hat{r})$ and next define the point $\overrightarrow{r_{j+1}}$ which is at the minimum distance from the origin to the level set of $\G=\G((r_{j}+h)\hat{r_j})$. Henceforth, by $\G(r_j), \G(r_j +h)$, we mean the values of the Green's function at the point $r_j \hat{r_j}$ and $(r_j +h)\hat{r_j}$ respectively , and we omit the unit vector from the notation. 

Note that using \cref{eq644}, and the subsequent argument till \cref{diffineq'}, first for the case $j=0$, that at each subsequent step we have, $\G(r_j +h)\leq \G(r_j)$ for each $j\geq 1$, and thus we have for each $j\geq 1$ that enables us to define $h_j\leq h$,  $h_j\leq h$ so that $\G(r_j +h)-\G(r_j)=\G(r_{j+1})-\G(r_{j+1} -h_{j})$. where $r_{j+1}-h_j$ is on the radial line joining the origin to $r_{j+1}$. Note that $\G(r_{j}+h)=\G(r_{j+1})$ by definition and thus $\G(r_{j})=\G(r_{j+1}-h_j)$.

 We have,
\begin{align}\label{growth}
    \G(r_j)=\G(r_j +h)-h\G^{'}(r_j +h) +\frac{h^{2}}{2}\G^{''}(r_j +h\theta),
\end{align}
for some $\theta<1$.

 This can be written, for some other constant $\theta_1<1$, 
\begin{align}\label{eq39}
    \frac{\G(r_j +h)-\G(r_j)}{h}=\G^{'}(r_j +h) -\frac{h}{2}\G^{''}(r_j)-\frac{h^{2}}{2}\theta\G^{'''}(r_j +\theta_1 h).
\end{align}
Note that, $\G(r_j +h)-\G(r_j)=\G(r_{j+1})-\G(r_{j+1} -h_{j})$, with $0\leq h_{j}\leq (r_{j+1}-r_j)\leq  h$ by definition.
Further, we have, 
\begin{align*}
    \G(r_{j+1} -h_{j})=\G(r_{j+1}) -h_{j}\G^{'}(r_{j+1}) +\frac{h_{j}^{2}}{2}\G^{''}(r_{j+1}) -\frac{h_{j}^{3}}{3!}\G^{'''}(r_{j+1} -\theta h).
\end{align*}
Using \cref{diffineq}, we further get the inequality,
\begin{align}
    \G(r_{j+1} -h_{j})-\G(r_{j+1})\leq -h_{j}\G^{'}(r_{j+1}) +\frac{h_{j}^{2}}{2}\Big(-\Big(\frac{n-1}{r_{j+1}}\Big) -\B(r_{j+1})\cdot \widehat{r_{j+1}}  \Big)\G^{'}(r_{j+1}) -\frac{h_{j}^{3}}{3!}\G^{'''}(r_{j+1} -\theta h).
\end{align}
This gives us, 
\begin{align}
     \G(r_{j+1} -h_{j})-\G(r_{j+1})\leq h_{j}\Big(-\G^{'}(r_{j+1})\Big)\Big(1 + \frac{n-1}{2r_{j+1}}h_{j} +\frac{1}{2}\B(r_{j+1})\cdot \widehat{r_{j+1}}h_{j}\Big)-\frac{h_{j}^{3}}{3!}\G^{'''}(r_{j+1} -\theta h_{j}).
\end{align}
This simplifies to,
\begin{align}
    \G^{'}(r_{j+1})\leq \frac{(-1)}{\big(1+ h_{j}(\frac{n-1}{2r_{j+1}} +\frac{1}{2}\B(r_{j+1})\cdot \widehat{r_{j+1}}) \big)}\Big(\frac{\G(r_{j+1} -h_{j})-\G(r_{j+1})}{h_{j}}+ \frac{h_{j}^{2}}{3!}\G^{'''}(r_{j+1} -\theta h_{j}) \Big).
\end{align}
Up to a first order approximation, this is written assuming that $h$ is arbitrarily small compared with $a_{i,k}$, that,
\begin{align}\label{eq43}
    \G^{'}(r_{j+1})\leq \big( 1 -  h_{j}(\frac{n-1}{2r_{j+1}}  +\frac{1}{2}\B(r_{j+1})\cdot \widehat{r_{j+1}}\Big)\big)\Big(\frac{\G(r_{j+1})-\G(r_{j+1} -h_{j})}{h_{j}}\Big)-\frac{h_{j}^{2}}{3!}\G^{'''}(r_{j+1} -\theta h_{j}).
\end{align}
Further note that,
\begin{align}\label{eq444}
    \G^{'}(r_j +h)=\G^{'}(r_j)+h\G^{''}(r_j)+\frac{h^{2}}{2!}\G^{'''}(r_j +\theta_2 h),
\end{align}
for some $\theta_2 <1$.
Using the fact that $\G(r_j  +h)-\G(r_j)=\G(r_{j+1})-\G(r_{j+1} -h_{j})<0$, and  $0\leq h_{j}\leq (r_{j+1}-r_j) \leq h$, and $h_j$ sufficiently small, we get from \cref{eq43}, that,
\begin{align}
    \G^{'}(r_{j+1})\leq \big( 1 -  h_{j}(\frac{n-1}{2r_{j+1}}  +\frac{1}{2}\B(r_{j+1})\cdot \widehat{r_{j+1}}\Big)\big)\Big(\frac{\G(r_{j}+h)-\G(r_{j})}{h}\Big)+\frac{h_{j}^{2}}{3!}\G^{'''}(r_{j+1} -\theta h_{j}).
\end{align}
Further using \cref{eq39,eq444}, we get, 
\begin{multline}
     \G^{'}(r_{j+1})\leq \big( 1 -  h_{j}(\frac{n-1}{2r_{j+1}}  +\frac{1}{2}\B(r_{j+1})\cdot \widehat{r_{j+1}}\Big)\big)\Big(\G^{'}(r_j)+\frac{h}{2}\G^{''}(r_j)+\frac{h^{2}}{2!}\G^{'''}(r_j +\theta_2 h) -\frac{h^{2}}{2}\G^{'''}(r_j +\theta_1 h)\Big)\\+\frac{h_{j}^{2}}{3!}\G^{'''}(r_{j+1} -\theta h_{j}).
\end{multline}
Thus we get,
\begin{multline}\label{eqimpp0}
     \G^{'}(r_{j +1})\leq \G^{'}(r_j)+  \frac{h}{2}\G''(r_j)     -  h_{j}(\frac{n-1}{2r_{j+1}}  +\frac{1}{2}\B(r_{j+1})\cdot \widehat{r_{j+1}}) \G'(r_j)+O\Big(\frac{h^{2}}{2} \G'''(r_j +\theta_1 h)+\frac{h_{j}^{2}}{2} \G'''(r_{j+1} +\theta_1 h_{j}) \\+ \frac{h^{2}}{2}\G''(r_j)(\frac{n-1}{2r_{j+1}}  +\frac{1}{2}\B\cdot \hat{r}) \Big),
\end{multline}
for some constants $\theta_1,\theta_2$, and so we have, choosing $L$ large enough so that we have $\frac{(n-1)}{r_{j+1}}\gg |\B(r_{j})\cdot \hat{r_j}|$, noting that $\G'(r_j)\leq 0$, $h_{j}\leq h$, and using \cref{diffineq},
\begin{multline}\label{growth1}
    \G^{'}(r_{j+1})\leq \G^{'}(r_j)-\frac{h}{2}\Big(\frac{n-1}{r_j} +\B(r_{j})\cdot \hat{r_j}\Big)\G^{'}(r_j) -  \frac{h}{2}\Big(\frac{n-1}{r_{j+1}}  +\B(r_{j+1})\cdot \widehat{r_{j+1}}\Big) \G'(r_j) + O(\frac{h^{2}}{2} \G'''(r_j +\theta_1 h)\\+\frac{h_{j}^{2}}{2} \G'''(r_{j+1} +\theta_1 h_{j})+ \frac{h^2}{r_{j+1}}\G''(r_j)).
\end{multline}
This gives us, 
\begin{multline}\label{growth3}
    \G'(r_{j+1})\leq \G'(r_j)-h\Big(\frac{n-1}{2}\Big(\frac{1}{r_j}+\frac{1}{r_{j+1}}\Big) +\frac{1}{2}(\B(r_{j})\cdot\hat{r_j}+\B(r_{j+1})\cdot\widehat{r_{j+1}})\Big)\G'(r_j)+O(\frac{h^{2}}{2} \G'''(r_j +\theta_1 h)\\+\frac{h_{j}^{2}}{2} \G'''(r_{j+1} +\theta_1 h_{j})+\frac{h^2}{r_{j+1}}\G''(r_j)).
\end{multline}
We note that $r_{j+1}- r_j \leq h$, and thus get,
\begin{multline}\label{eq64}
    \G'(r_{j+1})\leq \G'(r_j)-h\Big( \frac{n-1}{r_j}  
+ \frac{1}{2}(\B(r_{j})\cdot\hat{r_j}+\B(r_{j+1})\cdot\widehat{r_{j+1}})\Big)\G'(r_j) +O\Big(\frac{h^2}{2 r_{j}^2}\G'(r_j) +\frac{h^{2}}{2} \G'''(r_j +\theta_1 h)\\+\frac{h_{j}^{2}}{2} \G'''(r_{j+1} +\theta_1 h_{j})+\frac{h^2}{r_{j+1}}\G''(r_j)\Big).
\end{multline}
Note that the points $r_{j},r_{j+1}$ by construction could be far apart depending on the shape of the level sets of the Green's function.

Now also note by construction that, $0\leq h_j\leq r_{j+1}-r_j \leq  h$, and we have, 
\begin{multline}\label{eq63}
    \G(r_{j+1}-h_j)-\G(r_{j+1})=  -\G'(r_{j+1})h_j +\frac{h_{j}^2}{2}\G''(r_{j+1}+\theta_1 h_j)),\\
    \G(r_{j})-\G(r_{j}+h)= -\G'(r_{j})h +\frac{h^2}{2}\G''(r_{j}+\theta_2 h)),
\end{multline}
Thus we have,
\begin{align}\label{eq64}
     \G'(r_{j+1})h_j = \G'(r_{j})h -\frac{h_{j}^2}{2}\G''(r_{j+1}+\theta_1 h_j) + \frac{h^2}{2}\G''(r_{j}+\theta_2 h),
\end{align}
and so we have,
\begin{align}\label{eq66}
    |\G'(r_{j+1})|=\frac{h}{h_j}\big(|\G'(r_j)|+ \frac{h}{2}\G''(r_{j}+\theta_2 h) \big)+ \frac{h_{j}}{2}\G''(r_{j+1}+\theta_1 h_j)),
\end{align}
We consider two cases. First, we consider,
\begin{align} \label{eq6222}
(a): |\G'(r_j)|(r_{j+1}-r_j) \geq |\G'(r_j)|h_j\geq h|\G'(r_j)|- 2 h^2  \sup\limits_{K_k}|\G''(r)|
\end{align}
then, 
\begin{multline}\label{eq67}
    (r_{j+1} -r_j)\Big( \frac{n-1}{r_j}  
+ \frac{1}{2}(\B(r_{j})\cdot\hat{r_j}+\B(r_{j+1})\cdot\widehat{r_{j+1}})\Big)|\G'(r_j)|\\ \geq h \Big( \frac{n-1}{r_j}  
+ \frac{1}{2}(\B(r_{j})\cdot\hat{r_j}+\B(r_{j+1})\cdot\widehat{r_{j+1}})\Big)|\G'(r_j)| \\ -O(h^2 \sup\limits_{K_k}|\G''(r)|(\B(r_{j})\cdot\hat{r_j}+\B(r_{j+1})\cdot\widehat{r_{j+1}})\Big)
\end{multline}

Note that we can ensure the term $\Big( \frac{n-1}{r_j}  
+ \frac{1}{2}(\B(r_{j})\cdot\hat{r_j}+\B(r_{j+1})\cdot\widehat{r_{j+1}})\Big)$ is positive, by choosing $L$ large enough so that the $\frac{n-1}{r_j}  $ term is bigger in magnitude that the drift term, within the ball $B(0,1/L)$. We get, \cref{eq6222}, we have, $h\geq r_{j+1}-r_j \geq h -2h^2\frac{\sup\limits_{K_k}|\G''(r)|}{|\G'(r_j)|}$ and thus $(r_{j+1}-r_j)$ is comparable to $h$, by \cref{eq644}.

We have uniform bounds on the drift term in the domain $K_k$, so we can choose $h$ small enough so that the second term on the right is negligible at each step of the iteration in $K_k$, and thus we get from \cref{eq67,eq64}, noting that by construction, for each $j$, we have $\G'(r_j)<0$, that, 
\begin{multline}\label{eq68}
    \G'(r_{j+1})\leq \G'(r_j) - (r_{j+1} -r_j)\Big( \frac{n-1}{r_j}  
+ \frac{1}{2}(\B(r_{j})\cdot\hat{r_j}+\B(r_{j+1})\cdot\widehat{r_{j+1}})\Big)\G'(r_j) +O(h^2) 
\end{multline}

On the other hand, when we have the inequality opposite of \cref{eq6222}, we get, 
\begin{align}
(b):  |\G'(r_j)|h_j< h|\G'(r_j)|- 2 h^2  \sup\limits_{K_k}|\G''(r)|\Leftrightarrow  h|\G'(r_j)|>  |\G'(r_j)|h_j + 2 h^2  \sup\limits_{K_k}|\G''(r)|,
\end{align}
with $h$ small enough, we use the binomial approximation to first order to get from \cref{eq66},
\begin{align}
    |\G'(r_{j+1})|\geq  |\G'(r_j)|+ \frac{2h^2}{h_j}\sup\limits_{K_k} |\G''(r)|+  \frac{h^2}{2h_j}\G''(r_{j}+\theta_2 h) \big)+  \frac{h_{j}}{2}\G''(r_{j+1}+\theta_1 h_j)).
\end{align}
From here we easily see, for $h$ small enough, that,
\begin{align}\label{eq72}
    \G'(r_{j+1})\leq \G'(r_{j})
\end{align}

Further, the decay of the Green's function in either case is given for each $j$ by,
\begin{align}\label{growth2}
    \G(r_{j+1})=\G(r_j +h)=\G(r_j) +h\G^{'}(r_j) +\frac{h^2}{2}\G^{''}(r_j +\theta h).
\end{align}

At each step, we choose $h$ small enough, with $(r_{j+1}-r_j)\leq h$, so that we can write the right hand side of \cref{eq68} , 
\begin{multline}\label{eq73'}
    \G'(r_{j+1})\\  \leq \G'(r_j)\Big( 1- \frac{(n-1)(r_{j+1}-r_j)}{r_j} \Big) \Big( 1- (r_{j+1}-r_j) \frac{1}{2}(\B(r_{j})\cdot\hat{r_j}+\B(r_{j+1})\cdot\widehat{r_{j+1}})\Big)\Big) +O(h^2) \\ \leq  \G'(r_j)\Big( 1- \frac{(n-1)(r_{j+1}-r_j)}{r_j} \Big)\Big( e^{-(r_{j+1}-r_j) \frac{1}{2}(\B(r_{j})\cdot\hat{r_j}+\B(r_{j+1})\cdot\widehat{r_{j+1}})} \Big)+O(h^2).
\end{multline} 
Iterating this inequality $k$ times, we immediately see that the error term is bounded as $k O(h^2)$ for $\G'(r_{j+k})$.

We choose $N$ large enough so that $Nh\approx \Big(1/L -y_{i,\text{min}}\Big)\approx 1/L$. 

We note that for any two spheres $S(0, r_1), S(0,r_2)$ with $r_1<r_2<1/L$, we have that 
\begin{align}\label{sequential}
\text{min}_{S(0,r_2)}\G(\cdot)\leq \text{min}_{S(0,r_1)}\G(\cdot).\end{align}

To see this, consider the values of the Green's function at these minimum points, \\ $\text{min}_{S(0,r_1)}\G(\cdot)=\G(a_1),\text{min}_{S(0,r_2)}\G(\cdot)=\G(a_2)$ , where $a_1$ and $a_2$ need not be unique on the surface of the spheres $S(0,r_1), S(0,r_2)$ respectively. If we have to the contrary that $\text{min}_{S(0,r_2)}\G(\cdot)> \text{min}_{S(0,r_1)}\G(\cdot)$, then using the maximum prnciple we would conclude that $S(0,r_2)\subset \{x:\G(x)\geq \G(a_1)\}$ and then by definition the sphere $S(0,r_1)$ cannot be tangent to the hypersurface $\{x:\G(x)\geq \G(a_1)\}$ at $a_1$, and we have a contradiction.

In this case, the second order contribution in any of the derivatives considered in either of \cref{eq72,eq73'} at the $k$'th step with $k\leq N$ is $kO(h^2)$. Thus, further, the error in the quantity $|\G'(r_{j+k})|h\lesssim kO(h^3)$, for any $k$. Using \cref{sequential} and the argument following it, we see that the total error in the change of the Green's function is $\sum_{k=1}^{N}  k O(h^3)\approx N^2 O(h^3)\approx h\cdot  (1/L)^2 \to 0$ as $h\to 0$.


 
Thus the solution to the inequality
\begin{align}\label{eq855}
   \G'(r_{j+1})\leq \G'(r_j)\Big( 1- \frac{(n-1)(r_{j+1}-r_j)}{r_j} \Big)\Big( e^{-(r_{j+1}-r_j) \frac{1}{2}(\B(r_{j})\cdot\hat{r_j}+\B(r_{j+1})\cdot\widehat{r_{j+1}})} \Big),
\end{align}
as $h\to 0$, is given by the following differential inequality, 
\begin{align}\label{diffineq'}
    \frac{\partial^{2}\G(\cdot,0)}{\partial r^{2}} +\frac{(n-1)}{r}\frac{\partial \G(\cdot,0)}{\partial r}+ \B\cdot \hat{r} \frac{\partial \G(,0)}{\partial r} \leq 0,
\end{align}
 
\bigskip
\footnote{If we integrate this, to get for  $r\geq a_{i,k_0}$, that, 
\begin{align}\label{eq40}
    \frac{\partial \G}{\partial r}\Big|_{r} \leq C\frac{\partial \G}{\partial r}\Big|_{a_{i,k}}\Big(\frac{a_{i,k}^{n-1}}{r^{n-1}}\Big),
\end{align}

where the constant $C$ incorporates the exponential factor contribution from the drift term.}

Using the cases where the above differential inequality \cref{eq855} holds, along with the cases where we have \cref{eq72} holding, we thus have,  

\begin{lemma}For any $a_{i,k_0}\leq r_1\leq r_2$, with $r_1,r_2$ minimum points of some level sets of the Green's function, we get,
\begin{align}\label{eq41}
    \G(r_1,0)-\G(r_2,0)\geq 2C\Big| \frac{\partial \G}{\partial r}\Big|_{a_{i,k}}a_{i,k}^{n-1}\Big(\frac{1}{r_{1}^{n-2}} - \frac{1}{r_{2}^{n-2}} \Big).
\end{align}
\end{lemma}

\bigskip
Specifically we choose $a_{i,k}< r_1 =y_{i,\text{min}}\leq r_2=\frac{1}{L} $, and get,
\begin{align}\label{eq799}
   2s^{*}_{i}- \frac{1}{C_1}s^{*}_{i}\geq \frac{1}{C_1} \G(y_{i,\text{min}})-\G(\frac{1}{L}\widehat{s_{min}})\geq C \Big| \frac{\partial \G}{\partial r}\Big|_{a_{i,k}}a_{i,k}^{n-1}\Big(\frac{1}{y_{i,\text{min}}^{n-2}} -L^{n-2}\Big).
\end{align}

Here, $C_1$ is a uniform constant arising from the Harnack inequality. In case we have that $y_{i,\text{min}}\to 0$ as $i\to \infty$, then choosing $i$ a posteriori large enough, we get from above that,
\begin{align}\label{eq800}
    s^{*}_{i}\gtrsim C' \Big| \frac{\partial \G}{\partial r}\Big|_{a_{i,k}}a_{i,k}^{n-1}\cdot \frac{1}{y^{n-2}_{i,\text{min}}}
\end{align}
for some uniform constant $C'$.

\bigskip
Now instead of looking the points of minimum of the level sets of the Green's function, we look at the corresponding points of maximum of the level sets.

For these points of maximum, we have, 
\begin{align}\label{diffineq'''}
    \frac{\partial^{2}\G_i(\cdot,0)}{\partial r^{2}} +\frac{(n-1)}{r}\frac{\partial \G_i(\cdot,0)}{\partial r}+ \B\cdot \hat{r} \frac{\partial \G_i(\cdot,0)}{\partial r} \geq 0,
\end{align}


We first carry through a general calculation analogous to \cref{growth} through \cref{eq63}, for the maximum points. In this case, consider any maximum point $s_i \hat{s_i}$, where $\hat{s_i}=\frac{\vec{s_i}}{|s_i|}$,  and the level set $\G(\vec{s_i})$.  For any small enough $t$ to be determined later, consider the point $(s_i -t)\hat{s_i}$. and consider the level set of $\G((s_i -t)\hat{s_i})$ and the maximum point $\vec{s_{i+1}}$ of this level set. As before, we remove the reference to the radial unit vector from now on, for simplicity in notation.

We enumerate domains $\Omega_{m}:= \{x\in\Omega:\G(x)\leq 1/m  \}$, $m\geq 1$,  and take the sequence $m\to \infty$ in the final step of the argument. We also define, for each $m$,  $\Omega'_{m}:=\Omega_m \setminus B(0,b_{i,k_0})$, where $b_{i,k_0}\leq y_{i,\min}$ is a maximum point chosen from \cref{lemma}, with $a_{i,k_0}\leq y_{i,\text{min}}$ being the corresponding minimum point chosen from the lemma.

First we assume that, 
\begin{align}\label{eq977}
    \G(s_i -t)=\G(s_i) -t \G'(s_i) +\frac{t^2}{2}\G''(s_{i}-\theta t),
\end{align}
for some $\theta\leq 1$. In this case we choose $\beta_m:= 2\sup\limits_{\Omega'_m}|\G''(r)|$. We first consider the case, where $|\G'(s_i)|> K_m t$, for some $K_m \gg \beta_m$.\footnote{Note that this condition also appears prior to \cref{lemma8} later on.} In that case, it is clear from \cref{eq977} that we must have, 
\begin{align}
    \G(s_i -t)\geq \G(s_i).
\end{align}
We then define the length $t_i$ so that, $\G(s_{i+1} +t_i)=\G(s_i)$. By construction, we have, $0\leq t_i \leq s_i -s_{i+1}\leq t$. Further, we have, $\G(s_i -t )=\G(s_{i+1})$.  We can carry through calculations exactly analogous to \cref{growth} through \cref{eq63} here. 

The case where we have $|\G'(s_i)|\leq K_m t$ is considered in \cref{lemma8,lemma9} and the argument following \cref{lemma9}.

\begin{lemma}\label{lemma7}
Given a large positive integer $m$, for any level set $\{x:\G(x)=\G(\vec{s_i})\}$ with the value $\G(\vec{s_i})\geq 1/m$, with $|\G'(s_i)|> K_m t$, where $\vec{s_i}$ is the maximum point of this level set, with $|\vec{s_i}|\geq b_{i,k_0}$ , for $s_{i+1}$ defined above, with $t$ arbitrarily small in comparison to $b_{i,k_0}$ as well as arbitrarily small in comparison to $\inf \{\text{dist}(x,\partial\Omega): \G(x)=1/m\}$, we have,    
\begin{align}\label{eq670}
     \G^{'}(s_{i+1})\leq  \G^{'}(s_i) +\frac{t_{i}}{2}\Big(\frac{n-1}{s_{i+1}} +\B(s_{i+1})\cdot \widehat{s_{i+1}}) \Big)\G^{'}(s_i) + \frac{t}{2}\Big(\frac{n-1}{s_{i}} +\B(s_{i})\cdot \widehat{s_{i}}) \Big)\G^{'}(s_i) +O(t^2),
\end{align}
\end{lemma}
The condition that  $t$ be arbitrarily small in comparison to $b_{i,k_0}$ and $\inf \{\text{dist}(x,\partial\Omega): \G(x)=1/m\}$ ensures that $t/s_i$ and $t\B(s_{i})\cdot\widehat{s_{i+1}}$ are arbitrarily small. 
\begin{proof} We have, 
\begin{align}
    \G(s_i)=\G(s_i -t) +t\G'(s_i -t) +\frac{t^2}{2}\G''(s_i -t\theta').
\end{align}
for some $\theta' <1$. We write this, with some other constant $\theta_1$, that, 
\begin{align}\label{eq622}
    \frac{\G(s_i)-\G(s_i -t)}{t}= \G'(s_i -t) +\frac{t}{2}\G''(s_{i}) -\frac{t^2}{2}\theta'_1 \G'''(s_i -\theta_1 t).
\end{align}

Further, we get, 
\begin{align}
    \G(s_{i+1} + t_{i})=\G(s_{i+1}) +t_{i}\G^{'}(s_{i+1}) +\frac{t_{i}^{2}}{2}\G^{''}(s_{i+1}) +\frac{t_{i}^{3}}{3!}\G^{'''}(s_{i+1} +\theta t). 
\end{align}
Using \cref{diffineq'''}, we get from above that, 
\begin{align}
    \G(s_{i+1} +t_{i})-\G(s_{i+1})\geq t_{i}\G^{'}(s_{i+1}) +\frac{t_{i}^{2}}{2}\Big(-\Big(\frac{n-1}{s_{i+1}}\Big) -\B(s_{i+1})\cdot \widehat{s_{i+1}}  \Big)\G^{'}(s_{i+1}) +\frac{t_{i}^{3}}{3!}\G^{'''}(s_{i+1} -\theta t)\Big).
\end{align}
This gives us, 
\begin{align}
    \G(s_{i+1} +t_{i})-\G(s_{i+1})\geq t_{i}\Big(\G^{'}(s_{i+1})\Big)\Big(1 - \frac{n-1}{2s_{i+1}}t_{i} -\frac{1}{2}\B(s_{i+1})\cdot \widehat{s_{i+1}}t_{i}\Big)-\frac{t_{i}^{3}}{3!}\G^{'''}(s_{i+1} -\theta t_{i}).
\end{align}
\bigskip
This simplifies, noting that, $\G'(s_{i+1})\leq 0$, and that $t_i \leq t$ has been chosen arbitrarily small in comparison to $\inf\limits_{\Omega'_m}\text{dist}(x,\partial\Omega)$,
\begin{align}
    \G^{'}(s_{i+1})\leq \frac{1}{\big(1- t_{i}(\frac{n-1}{2s_{i+1}} +\frac{1}{2}\B(s_{i+1})\cdot \widehat{s_{i+1}}) \big)}\Big(\frac{\G(s_{i+1} +t_{i})-\G(s_{i+1})}{t_{i}}+\frac{t_{i}^{2}}{3!}\G^{'''}(s_{i+1} +\theta t_{i}) \Big).
\end{align}

Using the binomial approximation only up to first order approximation as earlier, we have, 
\begin{align}\label{eq666}
    \G^{'}(s_{i+1})\leq \big(1+t_{i}(\frac{n-1}{2s_{i+1}} +\frac{1}{2}\B(s_{i+1})\cdot \widehat{s_{i+1}}) \big)\Big(\frac{\G(s_{i+1} +t_{i})-\G(s_{i+1})}{t_{i}}+\frac{t_{i}^{2}}{3!}\G^{'''}(s_{i+1} +\theta t_{i}) \Big).
\end{align}

Further note,
\begin{align}\label{eq4444}
    \G^{'}(s_i -t)=\G^{'}(s_i)-t\G^{''}(s_i)+\frac{t^{2}}{2!}\G^{'''}(s_i -\theta_2 t),
\end{align}
for some $\theta_2 <1$. Using the fact that, $\G(s_i  -t)-\G(s_i)=\G(s_{i+1})-\G(s_{i+1} +t_{i})>0$, and that $0\leq t_i \leq s_i -s_{i+1}\leq t$, and $t_i\leq t$ sufficiently small, we get from \cref{eq666}, that, 
\begin{align}\label{eq667}
    \G^{'}(s_{i+1})\leq \big(1+t_{i}(\frac{n-1}{2s_{i+1}} +\frac{1}{2}\B(s_{i+1})\cdot \widehat{s_{i+1}}) \big)\Big(\frac{\G(s_{i} )-\G(s_{i}-t)}{t}+\frac{t_{i}^{2}}{3!}\G^{'''}(s_{i+1} +\theta t_{i}) \Big).
\end{align}
Using \cref{eq622}, we get, 
\begin{multline}\label{eq668}
    \G^{'}(s_{i+1})\leq \big(1+t_{i}(\frac{n-1}{2s_{i+1}} +\frac{1}{2}\B(s_{i+1})\cdot \widehat{s_{i+1}}) \big)\Big( \G'(s_i -t) +\frac{t}{2}\G''(s_{i}) -\frac{t^2}{2}\theta'_1 \G'''(s_i -\theta_1 t)\\ +\frac{t_{i}^{2}}{3!}\G^{'''}(s_{i+1} +\theta t_{i}) \Big).
\end{multline}
Now using \cref{eq4444}, we get, 
\begin{multline}\label{eq669}
    \G^{'}(s_{i+1})\leq \big(1+t_{i}(\frac{n-1}{2s_{i+1}} +\frac{1}{2}\B(s_{i+1})\cdot \widehat{s_{i+1}}) \big)\Big(\G^{'}(s_i)  -\frac{t}{2}\G''(s_{i})+\frac{t^{2}}{2!}\G^{'''}(s_i -\theta_2 t) \\ -\frac{t^2}{2}\theta'_1 \G'''(s_i -\theta_1 t)+\frac{t_{i}^{2}}{3!}\G^{'''}(s_{i+1} +\theta t_{i}) \Big).
\end{multline}
So we get, using also \cref{diffineq'''}, that, 
\begin{align}\label{eq113}
     \G^{'}(s_{i+1})\leq  \G^{'}(s_i) +\frac{t_{i}}{2}\Big(\frac{n-1}{s_{i+1}} +\B(s_{i+1})\cdot \widehat{s_{i+1}}) \Big)\G^{'}(s_i) + \frac{t}{2}\Big(\frac{n-1}{s_{i}} +\B(s_{i})\cdot \widehat{s_{i}}) \Big)\G^{'}(s_i) +O(t^2),
\end{align}
where the implied constant for the $O(\cdot)$ term depends on the domain $\{x:\G(x)\leq \G(s_i)\}$.
\end{proof}

For a fixed domain $\Omega'_{m}$, we enumerate the radial distances from $y_i$ to the points of maximum as $\{u_j \}_{j=1}^{\infty}$, with $u_{1}=b_{i,k_0}$ and the radial increments as $t$, where $t$ is arbitrarily small in comparison to $b_{i,k_0}$ as well as arbitrarily small in comparison to $d_m :=\min\{\text{dist}(x,\partial\Omega): x\in \Omega'_m\}$. 

Let $u_{j+1}$ be the point of maximum of the level set $\G(u_j +t)$. In case that $\G(u_j +t)<1/m$, we terminate with a value of $t^{*}\leq t$ and repeat the same argument. 

We then start the iteration from $u_{j+1}$, and iterate with decrements of length $t$, using \cref{lemma7}. We write $v_{j,1}=u_{j+1}$, and then $v_{j,2}$ being the point of maximum of the level set $\G(v_{j,1}-t)$. In this case, we have $\G(v_{j,2}+t_1)=\G(v_{j,1})$ and by definition $t_1\leq t$. For all $n\geq 2$, we define the $v_{j,n}$ and $t_n$
analogously. By definition, we have, $0\leq t_{n}\leq (v_{j,n}-v_{j,n+1}) \leq t$.

In this case, we get, by an argument similar to \cref{eq64}, that,
\begin{align}\label{eq80}
    |\G'(v_{j,n+1})|= \frac{t}{t_n}\Big(|\G'(v_{j,n})| +\frac{t}{2} \G''(v_{j,n} +\theta_2 t) \Big) -\frac{t_n}{2}\G''(v_{j,n+1} + \theta_1 t'_j ).
\end{align}
 We consider the two separate cases; first when 
\begin{align}\label{inq}
(a'):|\G'(v_{j,n})|t_n \leq  |\G'(v_{j,n})|t - \Big(\frac{(n-1)}{b_{i,k_0}} \Big)|\G'(v_{j,n})| + 2\sup\limits_{\Omega'_m}|\G''(r)|\Big)t^2.
\end{align}
In this case, we clearly have from \cref{eq80}, using that $t_{n}\leq t$, that, 
\begin{multline}\label{eq82}
  |\G'(v_{j,n+1})|\geq  |\G'(v_{j,n})|\Big( 1 + \Big(\frac{(n-1)}{b_{i,k_0}}\Big)t+\frac{M}{d_{m}^{1-\beta}}t \Big)  +2\sup\limits_{\Omega'_m}|\G''(r)|\frac{t^2}{t_n} +\frac{t^2}{t_n}\G''(v_{j,n} +\theta_2 t)\\ -\frac{t_n}{2}\G''(v_{j,n+1} + \theta_1 t'_j )  \geq  |\G'(v_{j,n})|\Big( 1 + \frac{(n-1)}{b_{i,k_0}} t  \Big). 
\end{multline}
It remains to consider the inequality opposite to \cref{inq}
\begin{multline}\label{eq83}
   (b'):  \Big(\frac{(n-1)}{b_{i,k_0}} + \frac{M}{d_{m}^{1-\beta}}\Big)|\G'(v_{j,n})| + 2\sup\limits_{\Omega'_m}|\G''(r)|\Big)t^2 + |\G'(v_{j,n})|t_n >  |\G'(v_{j,n})|t   \\  \Leftrightarrow -\Big(\frac{(n-1)}{b_{i,k_0}} + \frac{M}{d_{m}^{1-\beta}}\Big)|\G'(v_{j,n})| + 2\sup\limits_{\Omega'_m}|\G''(r)|\Big)t^2 + \G'(v_{j,n}) t_n < \G'(v_{j,n})t  
\end{multline}
For the sequence $v_{j,n}|_{n\geq 1}$ for this fixed $j$, we note that $t_n\leq (v_{j,n} -v_{j,n+1})\leq t$.

Here we have two separate cases, once when,
\begin{align*}
  (*):\   \Big(\frac{n-1}{2v_{j,n}}+ \frac{1}{2}\B(v_{j,n})\cdot \widehat{v_{j,n}}\Big)<0,
\end{align*}
when, substituting the expression for $t\G'(v_{j,n})$ from \cref{eq83}, we get, 
\begin{multline}\label{eq1199}
       \G^{'}(v_{j,n+1})\leq  \G^{'}(v_{j,n}) +t_n\Big(\frac{n-1}{2v_{j,n}}+ \frac{1}{2}\B(v_{j,n}\cdot \widehat{v_{j,n}})\Big)\G^{'}(v_{j,n})\\+ t_n\Big(\frac{n-1}{2v_{j,n+1}}+ \frac{1}{2}\B(v_{j,n+1})\cdot \widehat{v_{j,n+1}})\Big)\G^{'}(v_{j,n})  +O_m(t^2)\\ \leq \G^{'}(v_{j,n}) +(v_{j,n} -v_{j,n+1})\Big(\frac{n-1}{2v_{j,n}}+\frac{1}{2}\B(v_{j,n})\cdot \widehat{v_{j,n}}\Big)\G^{'}(v_{j,n}) \\+t_n\Big(\frac{n-1}{2v_{j,n+1}}+ \frac{1}{2}\B(v_{j,n+1})\cdot \widehat{v_{j,n+1}}\Big)\G^{'}(v_{j,n})  +O_m(t^2).
\end{multline}
On the other hand, as opposed to (*) when we have, 
\begin{align*}
  (**): \   \Big(\frac{n-1}{2v_{j,n}}+  \frac{1}{2}\B(v_{j,n})\cdot \widehat{v_{j,n}}\Big)\geq 0,
\end{align*}

then we get directly, using $-t \leq -(v_{j,n}-v_{j,n+1})$,
\begin{multline}\label{eq1122}
\G^{'}(v_{j,n+1}) \leq     \G^{'}(v_{j,n}) +(v_{j,n} -v_{j,n+1})\Big(\frac{n-1}{2v_{j,n}} + \frac{1}{2}\B(v_{j,n})\cdot \widehat{v_{j,n}}\Big)\\+ t_n\Big(\frac{n-1}{2v_{j,n+1}}+ \frac{1}{2}\B(v_{j,n+1})\cdot \widehat{v_{j,n+1}})\Big)\G^{'}(v_{j,n})  +O_m(t^2).
\end{multline}

In either case, from \cref{eq1122,eq1199}, we get the same expression.

Then performing a sign analysis again on the remaining coefficient, first consider the case where 
\begin{align*}
  (***): \   \Big(\frac{n-1}{2v_{j,n+1}}+  \frac{1}{2}\B(v_{j,n+1})\cdot \widehat{v_{j,n+1}}\Big)\geq 0,
\end{align*}

In this case, we again use \cref{eq83}, to absorb an additional second order term, and get immediately from \cref{eq1122}/\cref{eq1199}, that,
\begin{multline}
    \G^{'}(v_{j,n+1}) \leq     \G^{'}(v_{j,n}) +(v_{j,n} -v_{j,n+1})\Big(\frac{n-1}{2v_{j,n}} +\frac{n-1}{2v_{j,n+1}}\\ +\frac{1}{2}\B(v_{j,n})\cdot \widehat{v_{j,n}} + \frac{1}{2}\B(v_{j,n+1})\cdot \widehat{v_{j,n+1}})\Big)\G^{'}(v_{j,n})  +O_m(t^2).
\end{multline}

Lastly, when we have,
\begin{align*}
  (****): \   \Big(\frac{n-1}{2v_{j,n+1}}+  \frac{1}{2}\B(v_{j,n+1})\cdot \widehat{v_{j,n+1}}\Big)< 0,
\end{align*}

then we get directly from \cref{eq1199}/\cref{eq1122}, that,
\begin{multline}\label{eq1222}
    \G^{'}(v_{j,n+1}) \leq     \G^{'}(v_{j,n}) +(v_{j,n} -v_{j,n+1})\Big(\frac{n-1}{2v_{j,n}} +\frac{n-1}{2v_{j,n+1}}\\ +\frac{1}{2}\B(v_{j,n})\cdot \widehat{v_{j,n}} + \frac{1}{2}\B(v_{j,n+1})\cdot \widehat{v_{j,n+1}})\Big)\G^{'}(v_{j,n})  +O_m(t^2).
\end{multline}

\bigskip

Momentarily ignoring the second order term above, we get from \cref{eq1222} the reversed differential inequality, when $t\to 0$, that
\begin{align}\label{diffineq''}
    \frac{\partial^{2}\G_i(\cdot,0)}{\partial r^{2}} +\frac{(n-1)}{r}\frac{\partial \G_i(\cdot,0)}{\partial r}+ \B\cdot \hat{r} \frac{\partial \G_i(,0)}{\partial r} \geq 0,
\end{align}
This calculation was performed for a fixed $m$ as defined earlier. 

Define $\alpha_m:=\Big(\frac{(n-1)}{b_{i,k_0}} + \frac{M}{d_{m}^{1-\beta}}\Big)$, and $\beta_m:= 2\sup\limits_{\Omega'_m}|\G''(r)|$, then,  \cref{eq83} can also be written as, 
\begin{align}\label{eq115}
    (b''):|\G'(v_{j,n})|t_n \geq  |\G'(v_{j,n})|t - \alpha_m |\G'(v_{j,n})|t^2 - \beta_m  t^2.
\end{align}

In the case that $|\G'(u_{j+1})|=|\G'(v_{j,1})|\geq K_m t$, for some $K_m$, dependent on $m$, to be chosen sufficiently large in comparison to $\beta_m$, we get for $r=1$, from the condition (b'') of  \cref{eq115}, that, 
\begin{align}\label{eq1155}
    t_1 \geq t\big(1- \alpha_m t -\frac{\beta_m}{|\G'(v_{j,1})|} t\big) \geq t\big( 1 -\frac{\beta_m}{K_m} - \alpha_m t) \big).
\end{align}
On the other hand, when,  $|\G'(u_{j+1})|=|\G'(v_{j,1})|\leq K_m t$, then $|\G'(u_{j+1})|t\leq K_{m}t^2$. 

For any $n\geq 1$, in the situation of \cref{eq115}, we get as deduced in \cref{eq1122}, that, 
\begin{multline}\label{eq11222}
|\G^{'}(v_{j,n+1})| \geq    | \G^{'}(v_{j,n})| +(v_{j,n} -v_{j,n+1})\Big(\frac{n-1}{v_{j,n}}+ \frac{1}{2}\B(v_{j,n+1})\cdot \widehat{v_{j,n+1}}+ \frac{1}{2}\B(v_{j,n})\cdot \widehat{v_{j,n}}\Big)|\G^{'}(v_{j,n}) | \\ +O_m(t^2).
\end{multline}
On the other hand, note that in the corresponding case (a') of \cref{eq82}, there is no $O_m(t^2)$ term on the right side of the inequality.\footnote{Note that in $(a')$ in \cref{inq}, in order to get sufficient decay of $|\G'(v_{j,n})|$ in comparison to $|\G'(v_{j,n+1})|$, it is sufficient to just consider the summand $(n-1)/b_{i,k_0}$ and neglect the term $M/d_{m}^{1-\beta}$.}

We define the Riemann sum, $\Delta_{t}(u_j, u_{j+1})=\sup\limits_{w(t,u_j,u_{j+1})}\sum\limits_{p=0}^{N(t,u_j,u_{j+1})} t |\B(w_{p})|,$ where the supremum is taken over all sequences $w(t,u_j,u_{j+1})$,  where $w_{0}=u_j$ and $w_{N(t,u_j,u_{j+1})}=u_{j+1}$, and for each $p\geq 1$, we have $|w_p|=|w_{p-1}|+t$.

Then define $\Delta(u_j,u_{j+1})=\limsup \limits_{t\to 0} \Delta_{t}(u_j, u_{j+1})$. More generally, one defines for any $a,b\in \Omega$ with $|a|<|b|$, the term $\Delta(a,b)=\limsup \limits_{t\to 0} \Delta_{t}(a, b)$ in an identical manner.

Now, we show,
\begin{lemma}\label{lemma8}
    For all $t$ arbitrarily small in comparison to $\alpha^{-1}_m, \beta_{m}^{-1}$, when $|\G'(v_{j,1})|= |\G'(u_{j+1})|> K_m t $, for any $j\geq 1$, with $K_m\gg \beta_m$, there exist constants  $c, \Delta(\Omega)$, so that, $K'_m= e^{-\Delta(\Omega)}K_m$,  so that for all $n\geq 2$, we have $|\G'(v_{j,n})|> K'_m t $.
\end{lemma}
We have to ensure that through successive iterations of \cref{eq11222} , the $O_m(t^2)$ term does not dominate the main term on the right of \cref{eq11222}.
\begin{proof}[Proof of \cref{lemma8}]
  Note that for any $n\geq 1$, whenever we are in the regime of case (a') of \cref{inq}, then from \cref{eq82}, we clearly get $|\G'(v_{j,n+1})|>|\G'(v_{j,n})|\Big( 1 + \frac{(n-1)}{b_{i,k_0}} t +\frac{M}{d_{m}^{1-\beta}}t \Big)$. In the case we are in the regime of (b') and thus \cref{eq11222}, first note that, $|v_{j,r}|-|v_{j, r+1}|\geq t_r$, for any $r\geq 1$. For any $1\leq r<n$, using induction, we note that in the regime of (b''), the condition of \cref{eq1155} holds with $K_m$ replaced with $K'_m$, i.e., $t\geq t_r \geq t\big(1- \beta_m/K'_m - t\alpha_m \big)$ . Thus $t_r$ is approximable by $t$ when $\beta_m /K'_m \ll 1$, and the number of times the $O_m(t^2)$ term is applied for the $j'$th iteration, is upper bounded by $\lesssim N_{j,n} <  (|u_{j+1}|-|v_{j,n}|)/t_r \approx  (|u_{j+1}|-|v_{j,n}|)/t $ where one ignores corrections of the order of $O(t^3)$ or higher. 

Thus, the total contribution of the error term is bounded by $N_{j,n} O_m(t^2)\lesssim (\text{diam}\Omega)t$. The main contribution now comes from either \cref{eq82} or the main term of the right of \cref{eq11222}, and one gets, when $t$ is small enough, by using the same argument as in \cref{eq73'}(with the reverse inequality), noting that $|v_{j,1}|>|v_{j,n}|$, 
\begin{multline}
    |\G'(v_{j,n})|\geq |\G'(v_{j,1})|\frac{|v_{j,1}|^{n-1}}{|v_{j,n}|^{n-1}}e^{-c\Delta(v_{j,n}, v_{j,1})} -C (\text{diam}\Omega)t +O_m(t^2) \\ \geq K_m t e^{-c\Delta(v_{j,n}, v_{j,1})} -C (\text{diam}\Omega)t +O_m(t^2)> \frac{1}{2}K_m t e^{-c\Delta(v_{j,n}, v_{j,1})} .
\end{multline}
We note that the drift is integrable, and so we have, $\Delta(a,b)<\Delta$ for any $a,b\in \Omega$ and for some uniform $\Delta\geq 0$, and it is enough to choose $K_m'\approx K_{m} e^{-c\Delta}$. 
\end{proof}

We can show that,
\begin{lemma}\label{lemma9}
    When $|\G'(v_{j,1})|= |\G'(u_{j+1})|\geq K_m t $, we get,
    \begin{align}\label{eq40''}
   \Big| \G'(u_{j}) \Big|\geq \Big|\G'(u_{j+1})\Big| \Big(\frac{|u_{j+1}|^{n-1}}{|u_{j}|^{n-1}}\Big)e^{-c\Delta(u_j,u_{j+1})}.
\end{align}
\end{lemma}
\begin{proof}
    This result will follow from repeatedly using \cref{eq82,eq11222}.

Recall that, $v_{j,1}=u_{j+1}$, and for each $r\geq 2$,  $v_{j,r+1}$ is the point of maximum of the level set $\G(v_{j,r}-t)$.
Hypothetically one of the two following cases happen:
\begin{itemize}\label{cases}
    \item The values $|v_{j,n}|$ converge, as $n\to \infty$, to some value $v_{j}$ with $|v_j|> |u_{j}|$, or
    \item  For some $n_0$, we have that $|v_{j,n_0}|-|u_{j+1}|=t^{**}\leq t$. 
   \end{itemize} 
    
     In the case where the sequence $|v_{j,n}|$ converges, as $n\to \infty$, and thus also a subsequence $v_{j,n_k}\to v_j$ for some value $v_{j}$ with $|v_j|> |u_{j}|$, then we are in the situation of \cref{eq82} for all large enough $k$, and thus either $\G'(v_{j,n_k})\to \infty$ as $k\to \infty$, which is a contradiction to the fact that $\G$ is locally in $C^{3,\alpha}$ outside the pole of the Green's function, or $\G'(u_{j+1})=0$ and we would be done by contradiction.

Otherwise in the second case, by using \cref{eq82} or the main term on the right hand side of \cref{eq11222}, whichever is applicable in each step depending on whether condition (a) or (b) is satisfied, noting by \cref{lemma8} that the total contribution of the $N_j O_m (t^2)\to 0$ as $t\to 0$, where $N_j \lesssim (|u_{j+1}|-|u_j|)/t$, we get the claimed result.
\end{proof}
Note that for each $j\geq 1$,  we have $|u_{j+1}|\geq|u_j|+t$, by construction.  Also, note that as $m$ gets large and thus $d_m \to 0$, the implied constant in the $O_m (t^2)$ term gets large, and for any fixed $m$, we choose the increment $t$ small enough.

We start the iteration from the maximum point on the sphere $S(0,1/L)$, and consider the first $j_0 \geq 1$ so that $|\G'(u_{j_0})|\leq K'_m t$. Then by the arguments in the proof of \cref{lemma8,lemma9}, we get that for all subsequent $j\geq j_0$, $|\G'(u_j)|\leq  K_m t e^{\Delta} $. In this case, as $t\to 0$, $\G(u_{j}+t)-\G(u_{j})\lesssim K'_{m}e^{\Delta}t^2= K_{m}t^2$ keeping  the contribution up to the second order in $t$. Starting on the sphere $S(0, 1/L)$ , till the level set $\G(\cdot)=1/m$, we have $N_m$ steps in the iteration where by a crude bound, $N_{m}t\leq \text{diam}\Omega$, and thus, the total change in the Green's function for $j\geq j_0$, is bounded by $C N_m e^{\Delta} K_m t^2\leq (\text{diam}\Omega) K_m t \to 0$ as $t\to 0$. For all $j\leq j_0$, the result of \cref{lemma9} holds.

The total change in the Green's function is given using the Taylor theorem, by
\begin{align}\label{eq1300}
   \sum\limits_{j=1}^{J} \G(u_{j})-\G(u_{j}+t)\approx \sum\limits_{j=1}^{J}\Big(|\G'(u_{j})|t + O(t^2)\Big),
\end{align}
where for each $j$, we have $\G(u_j +t)=\G(u_{j+1})$, $|u_1|=  1/L$, and $\G(u_J)=1/m$.
This, in the limit as $t\to 0$, is bounded from above by,
\begin{align}
    \leq C'\Big| \frac{\partial \G}{\partial r}\Big|_{ b_{k}}b_{i,k}^{n-1}\Big(L^{n-2} - \frac{1}{r_{i,m}^{n-2}} \Big),
\end{align}
where $r_{i,m}$ is the point which is the maximum distance from the origin on the level set $\G(\cdot)=1/m$, and by definition, we have $\G(r_{i,m})=1/m$. Note that the second order contribution in \cref{eq1300} goes to $0$ when in the last step we take the limit of $t\to 0$.

After this step, we can take $m\to \infty$. 
Then in the limit of $m\to \infty$, we have that, 
\begin{multline}\label{eq900}
s^{*}_{i}=\G\Big(\frac{1}{L}\widehat{s_{max}},0\Big)   = \G \Big(\frac{1}{L}\widehat{s_{max}},0\Big)-\G(r_{i},0)\leq C'\Big| \frac{\partial \G}{\partial r}\Big|_{ b_{i,k}}b_{i,k}^{n-1}\Big(L^{n-2} - \frac{1}{r_{i,m}^{n-2}} \Big),\\ \leq C'\Big| \frac{\partial \G}{\partial r}\Big|_{ b_{i,k}}b_{i,k}^{n-1}L^{n-2} ,.
\end{multline}
where $r_{i}$ is the point which is the maximum distance from $y_i$ to the boundary, and by definition, we have $\G(r_{i},0)=0$ . 

Now, we will compare \cref{eq900} and \cref{eq800}. For this, we will use \cref{lemma}. Given any $i$, we choose $k$ large enough so that we have two points $a_{i,k},b_{i,k}\leq y_{i,\text{min}}$ satisfying \cref{lemma}. Now comparing the right hand sides of  \cref{eq800,eq900}, we conclude that $y_{i,\text{min}}$ can't be arbitrarily small, thus establishing \cref{imp}, and Case (i)  is established.
\bigskip



\item[Case(ii).] The proof for Case (ii) follows with the same essential argument as that for Case(i) outlined above. In this case, for each large enough integer $N$, we have $f(s_N)>N f(2s_N)$ for some $s_N \geq \widetilde{s_{N}}$.  Then the preceding condition implies that there exist boundary points $z_{N}\in \partial\Omega_{N}, z_{2N}\in \partial\Omega_{2N}$  with $z_{2N}$ being the point at the minimum distance to $\partial\Omega_{2N}$ from the origin, and $z_N$ being the point at the maximum distance to $\partial\Omega_{N}$, and the ratio
\begin{align}\label{eq123123}|z_{N}|/|z_{2N}|\to \infty \  \text{as} \ N\to \infty\end{align}

In this case, consider the sphere $S(0,|z_N|)$ of radius $|z_N|$, which by definition is contained inside the sphere $S(0,1/2)$ since we have  $s_N \geq \widetilde{s_{N}}=\max\limits_{S(0,1/2)}\G(x,0)$. Then analogous to the earlier case, consider the point $s^{*}=\max\limits_{x\in S(0,|z_N|)}\G(x,0)=\G(z_N,0)$, by definition of the point $z_N$ being the point at the maximum distance to $\partial\Omega_N$ from the origin.

In this case, one argues identically as in Case(i), looking at the points of maximum and minimum of the level sets of the Green's function for the proof of \cref{imp}, using essentially again \cref{eq900,eq800}, using a sequence of points for each fixed $N$, labelled  $a_{N,k}, b_{N,k}|_{k=1}^{\infty}$ converging to $0$ as $k\to \infty$ satisfying the condition of \cref{lemma} 

This would give us, similar to \cref{eq900}, that, 
\begin{align}
    s_{N}^{*}\leq C'\Big| \frac{\partial \G}{\partial r}\Big|_{ b_{N,k}}b_{N,k}^{n-1}\frac{1}{|z_N|^{n-2}}.
\end{align}
Further, anaogous to \cref{eq800}, we would also have for some $C_1 >1$ that,
\begin{align}
    2s^{*}_N -\frac{1}{C_1} s^{*}_N \gtrsim C \Big| \frac{\partial\G}{\partial r}\Big|_{a_{N,k}}a_{N,k}^{n-1}\Big(\frac{1}{|z_{2N}|^{n-2}} -\frac{1}{|z_N|^{n-2}}\Big)
\end{align}
\end{enumerate}
Using these two equations above, as well as \cref{lemma}, using \cref{eq123123}, we get a contradiction as $N\to \infty$.

This concludes the proof of \cref{thm2}.
\end{proof}
\begin{remark}
We may have some $u$ with $|u|<|u_{j_0}|$ (Note the definition of $u_{j_0}$ in the paragraph prior to \cref{eq1300}) for which $|\G'(u)|=O(t)$ and consequently the contributions from the two summands in \cref{eq1300} are comparable with a contribution of $O(t^2)$. We have a sum of $O(1/t)$ terms, and thus a net contribution of the order of $O(t)$, which in the final limit of $t\to 0$, vanishes.
\end{remark}

We now prove \cref{thm3}.
\begin{proof}[Proof of \cref{thm3}]
    We use the pointwise upper bound on $\G(x,0)$, the Harnack principle for the solutions for $L$ and $L_T$, and a standard argument using the Green's functions for balls contained in B(0,1) and the maximum principle. We also use the standard fact mentioned in the introduction that $\G(x,y)=\G_T(y,x)$ where $G_T$ is the Green's function corresponding to the adjoint operator.

    Consider without loss of generality any point $w\in S(0,r)$. Call $\delta(w)=d_w$. Then consider the integer $m$ so that $2^m d_w \approx 1$. Then one can choose a sequence of $2m+1$ many points $\{0=w_0, w_1,\dots, w_{2m}=w\}$, where for each $0\leq i\leq 2m$, the point $w_i$ lies on the line joining $0$ and $w$, and for each $m$, we have $|w_{2i} -w_{2i-1}|\approx |w_{2i +1}-w_{2i}|\approx 2^{m-i} d_w$ where the implied constant is independent of $m$.
    
   By repeated use of Harnack's inequality, we have $\G(w_1,0)=\G(w_1,w_0)=\G_T (w_0,w_1)=K\G_T(w_2,w_1)=K\G(w_1,w_2)=K^2\G(w_3,w_2)=K^4 \G(w_5,w_4)=\dots=K^{2t}\G(w_{2t+1},w_{2t})=\dots$.

 Note that from \cref{thm2}, we have bounds for $\G(w_1,w_0)$.  Thus, by another use of the Harnack inequality, we get, for any $x\in S(w,\frac{1}{2}\delta(w))$, with the constant $K(m)$ dependent on $m$, that 
    \begin{align}
        \G(x,w)\leq K(m).
    \end{align}

Now, consider the Dirichlet Green's function $\G_{ball}(\cdot,w)$ for the ball $B(w,\frac{1}{2}\delta(w))$.

Then we have uniform estimates on $|\B|\leq M'(m)$ in this ball $B(w,\frac{1}{2}\delta(w))$. Then by arguments such as in \cite{KimSa,Morg} we have interior estimates for $\G_{ball}$, 
\begin{align}
    \G_{ball}(x,w)\leq \frac{K'(M'(m))}{|x-w|^{n-2}}\ \text{when} \  x\in B(w,\frac{1}{4}\delta(w)).
\end{align}
Then by the maximum principle, we have, noting that by definition  $\G_{ball}(x,w)=0$ when $x\in \partial B(w,\frac{1}{2}\delta(w))$, that,
\begin{align}
    \G(x,w)-\G_{ball}(x,w)\leq K(m), \ \ \text{when} \ \  x\in B(w,\frac{1}{2}\delta(w)).
\end{align}
In particular, for $x\in  B(w,\frac{1}{4}\delta(w))$, we have,
\begin{align}
    \G(x,w)\leq K(m) + \frac{K'(M'(m))}{|x-w|^{n-2}}. 
\end{align}
Now consider the length $r_w$(dependent on $m$) so that $ r_{w}^{n-2}= \frac{K'(M'(m))}{K(m)}$ In that case, we see that for any $x\in B(w,r_w)$, we have,
\begin{align}
    |x-w|^{n-2} <r_w ^{n-2}=\frac{K'(M'(m))}{K(m)} \Leftrightarrow K(m)<\frac{K'(M'(m))}{|x-w|^{n-2}},
\end{align}
and thus we have, 
\begin{align}
    \G(x,w)\leq \frac{2K(m)}{|x-w|^{n-2}}, \ \ \text{when}\ x\in B(w,r_w).
\end{align}
If we have $r_w\leq \frac{1}{4}\delta(w)$,  then by using Harnack inequality, we extend the inequality to the entire ball $B(w,\frac{1}{2}\delta(w))$, using Harnack inequality, and noting that the constant is altered by a factor dependent on $m=-\frac{\log (1-r)}{\log 2}$, thus being dependent on $r$.

In case $r_w >\frac{1}{4}\delta(w)$, then for any $x\leq \frac{1}{4}\delta(w)$, we already have the result in the ball $B(w,\frac{1}{4}\delta(w))$ which we also extend to $B(w,\frac{1}{2}\delta(w))$ using Harnack's inequality.

    For points in  $w\in B(0,r)$, we essentially repeat the same argument with a fewer number of steps and better constants. 
    This concludes the proof. 
\end{proof}
\begin{remark}
    The geometry of the unit ball is essential in the proof of the result, in Lemma 8, when the drift is unbounded in the generality considered here. In ongoing work, we are extending these results in more general bounded domains. We will also study adequate perturbations of the principal Laplacian term for which this construction holds.
\end{remark}



\section{Conclusion and future directions.}
In future work, we will consider more general bounded Lipschitz domains for which existence of solutions to this operator is guaranteed by the result of \cite{Ha24}.


We also intend to use the method of this manuscript, for the time independent Schroedinger equation and study newer classes of potentials for which one is able to obtain pointwise bounds on the Green's function. Specifically, one anticipates to first get a condition of the form in \cref{eq171} and continue the argument with appropriate modifications.  In that case, one would require a similar level of regularity on the potential term as we required for the drift term here. 

The argument presented here for the drift term does not immediately generalize to this broader set of operators with elliptic terms in divergence or non-divergence form. Crucially, the spherical form of the Laplacian is used in this proof. Specifically, a version of \cref{diffineq} does not hold for a general elliptic term in place of the Laplacian. We will study perturbations of the Laplacian in our operator, in the future. We will also consider the Green's function corresponding to the parabolic operator with such drifts and potentials.


\section{Acknowledgements.} The author is grateful to Stephen Montgomery-Smith for many hours of extensive discussions, and to Seick Kim for feedback on this problem.


\bigskip

Aritro Pathak, Indian Statistical Institute, Kolkata, India.

Email: ap7mx@missouri.edu.


\begin{thebibliography}{}
\footnotesize



\bibitem[AiH08]{AiH08}Aikawa, H. and Hirata, K., 2008. Doubling conditions for harmonic measure in John domains. In \textit{Annales de l'institut Fourier} (Vol. 58, No. 2, pp. 429-445). 

\bibitem[An86]{An86} Ancona, Alano. "On strong barriers and an inequality of Hardy for domains in Rn." Journal of the London Mathematical Society 2, no. 2 (1986): 274-290.

\bibitem[CFMS81]{CFMS} Cafarelli, L., Fabes, E., Mortola, S., and Salsa, S. (1981). Boundary behavior of non negative solutions of elliptic operators in divergence form, Indiana J. Math., 30 (1981), 621, 640.

\bibitem[Ch91]{Ch91} Christ, Michael. "On the equation in weighted L 2 norms in C1 equation in weighted L 2 norms in C1." The Journal of Geometric Analysis 1, no. 3 (1991): 193-230.

\bibitem[Ev22]{Evans} Evans, Lawrence C. \textit{Partial differential equations}. Vol. 19. American Mathematical Society, 2022.


\bibitem[GT77]{Gt} Gilbarg, David, Neil S. Trudinger, David Gilbarg, and N. S. Trudinger. Elliptic partial differential equations of second order. Vol. 224, no. 2. Berlin: springer, 1977.

\bibitem[Gra23]{Grafakos} Grafakos, Loukas. Fundamentals of Fourier Analysis. Springer, 2023.

\bibitem[GW82]{GrWi} Grüter, M. and Widman, K.O., 1982. The Green function for uniformly elliptic equations. Manuscripta mathematica, 37(3), pp.303-342.

\bibitem[Ha24]{Ha24}Hara, Takanobu. "Global H\"{o} lder solvability of second order elliptic equations with locally integrable lower-order coefficients." arXiv preprint arXiv:2403.01104 (2024).

\bibitem[HL01]{Singdr} Hofmann, S., and Lewis, J. L. (2001). The Dirichlet problem for parabolic operators with singular drift terms (Vol. 719). American Mathematical Society.

\bibitem[HLMP22]{hlmp22}Hofmann, Steve, Linhan Li, Svitlana Mayboroda, and Jill Pipher. "The Dirichlet problem for elliptic operators having a BMO anti-symmetric part." \textit{Mathematische Annalen} 382, no. 1-2 (2022): 103-168.

\bibitem[HMMTZ21]{HMMTZ21}Hofmann, Steve, José María Martell, Svitlana Mayboroda, Tatiana Toro, and Zihui Zhao. "Uniform rectifiability and elliptic operators satisfying a Carleson measure condition." \textit{Geometric and Functional Analysis} 31, no. 2 (2021): 325-401.

\bibitem[KP01]{KP01}Kenig, Carlos E., and Jill Pipher. "The Dirichlet problem for elliptic equations with drift terms." \textit{Publicacions matemàtiques} (2001): 199-217.

\bibitem[KS19]{KimSa} Kim, S., and Sakellaris, G. (2019). Green’s function for second order elliptic equations with singular lower order coefficients. Communications in Partial Differential Equations, 44(3), 228-270.

\bibitem[MM22]{MM22} Maz'ya, Vladimir, and Robert McOwen. "The Fundamental Solution of an Elliptic Equation with Singular Drift." arXiv preprint arXiv:2209.00058 (2022).

\bibitem[Moo15]{Moo15}Mooney, C., 2015. Harnack inequality for degenerate and singular elliptic equations with unbounded drift. \textit{Journal of Differential Equations}, \textit{258}(5), pp.1577-1591.

\bibitem[Mor19]{Morg} Mourgoglou, Mihalis. "Regularity theory and Green's function for elliptic equations with lower order terms in unbounded domains." \textit{arXiv preprint arXiv:1904.04722} (2019).

\bibitem[Na58]{Na58}Nash, John. "Continuity of solutions of parabolic and elliptic equations." \textit{American Journal of Mathematics} 80, no. 4 (1958): 931-954.

\bibitem[Naz12]{Naz12}Nazarov, A. "The Harnack inequality and related properties for solutions of elliptic and parabolic equations with divergence-free lower-order coefficients." \textit{St. Petersburg Mathematical Journal} 23, no. 1 (2012): 93-115.

\bibitem[Pat24]{Pat24} Pathak, Aritro. "A counterexample for pointwise upper bounds on Green's function with a singular drift at boundary." arXiv:2405.13313.

\bibitem[Pog24]{Pog24} Poggi, Bruno. "Applications of the landscape function for Schrödinger operators with singular potentials and irregular magnetic fields." \textit{Advances in mathematics} 445 (2024): 109665.

\bibitem[Sak21]{Sake} Sakellaris, Georgios. "On scale-invariant bounds for the Green’s function for second-order elliptic equations with lower-order coefficients and applications." \textit{Analysis \& PDE} 14, no. 1 (2021): 251-299.

\bibitem[She94]{She94} Shen, Zhongwei. "On the Neumann problem for Schrödinger operators in Lipschitz domains." Indiana University Mathematics Journal (1994): 143-176.

\bibitem[She95]{She95} Shen, Zhongwei. "$ L^ p $ estimates for Schrödinger operators with certain potentials." In Annales de l'institut Fourier, vol. 45, no. 2, pp. 513-546. 1995.

\bibitem[She99]{She99} Shen, Zhongwei. "On fundamental solutions of generalized Schrödinger operators." Journal of Functional Analysis 167, no. 2 (1999): 521-564.


\bibitem[SSSZ12]{sssz12} Seregin, Gregory, Luis Silvestre, Vladimír Šverák, and Andrej Zlatoš. "On divergence-free drifts." \textit{Journal of Differential Equations} 252, no. 1 (2012): 505-540.


\bibitem[Wi24]{Wi} Wiener,Norbert. ``The Dirichlet Problem''. Journal of Mathematics and Physics, 3 (1924), no. 3, 127–146.


\end{thebibliography}
\end{document}